\documentclass[12pt]{article}

\usepackage{amsmath,amsfonts,amssymb,amsthm,amscd}
\usepackage[latin1]{inputenc}
\usepackage{graphicx}
\usepackage[dvips]{epsfig}
\usepackage{graphics}
\usepackage{psfrag}
\usepackage{amsmath}
\usepackage{amsthm}
\usepackage{amssymb}
\usepackage{latexsym}

\input trans

\def\LL{{\cal L}}
\def\ll{\mbox{l\hspace{ -.15em}L}}

\def\RR{\mathbb{R}}
\defß{{\bf S}}

\newcommand{\EE}{\mathbb{E}}

 \def\Scp#1#2{\big({#1}\,\big|\,{#2}\big)}

 \def\ppt#1{{\overset{\mbox{\large .\kern-1pt.}}{#1}}}
 \def\pppt#1{{\overset{\mbox{\large .\kern-1pt.\kern-1pt.}}{#1}}}

 \def\Pt#1{{\overset{\mbox{\Huge .}}{#1}}}
 \def\Ppt#1{{\overset{\mbox{\Huge .\kern-3pt.}}{#1}}}
 \def\Pppt#1{{\overset{\mbox{\Huge .\kern-3pt.\kern-3pt.}}{#1}}}

\def\bs{{\bf S}}

\defß{\mbox{\bf S}}
\def\fol{{\cal F}}

\def\vect#1{{\overset{\rightarrow}{#1}}}

\usepackage{color}

\newcommand{\spa}{\mathrm{span}}

\newtheorem{theo}{Theorem}
\newtheorem{lemm}[theo]{Lemma}

\newtheorem{defi}[theo]{Definition}
\newtheorem{prop}[theo]{Proposition}

\newenvironment{rema}{{\noindent \bf Remark:}}{}
\newenvironment{demo}{{\noindent \bf Proof:}}{\hfill$\Box$}

\title{\bf Dynamical behavior of Darboux curves}
\author{Ronaldo Garcia, Rémi Langevin and Pawel Walczak}

\begin{document}
\maketitle

   \begin{abstract} In 1872 G. Darboux defined a family of curves on surfaces of $\mathbb R^3$ which are preserved by the action of the Möbius group and share many properties with geodesics. Here we   characterize these  curves under the view point of Lorentz geometry and prove some general properties and make them explicit them on simple surfaces, retrieving 
 results of Pell (1900) and Santaló (1941).    \end{abstract}

\vskip 5 mm

\section*{Introduction}

Our interest here is to understand a family of curves on a surface called {\it Darboux curves}. Almost as in the case of geodesics, through every point and direction which is not of principal curvature,  passes a unique Darboux curve. References about  these curves are \cite{Da}, \cite{Ri}, \cite{Co}, \cite{Pe}, \cite{En}, \cite{Sa}, \cite{sa2}, \cite{Se}.

To understand the dynamics of these curves we will   give a natural family of spheres along a  Darboux curve in the set of spheres of $\mathbb R^3$ or $ß^3$. Then we will study the angle drift of Darboux curves with respect to the foliations of the surface by principal curvature lines and by curves making a constant angle with the principal curvature foliations.

We will also study the Darboux curves on some special surfaces such as the envelopes of very particular one-parameter family of spheres, and on quadrics.

This  paper is organized as follows: In section \ref{sc:1} is described the space of spheres of $ß^3$ and the  correspondence with  a quadric $\Lambda\subset \mathbb R^5$ with the induced Lorentz metric of signature 1.
In section \ref{sc:5} a comparative study of quadrics and Dupin cyclides is studied considering a foliation by constant angle with the principal lines.
 In section
\ref{sc:2} is the introduced the space of osculating spheres to a surface and the Darboux curves are   characterized geometrically.  In section \ref{sc:3} is performed a  geometric study of Darboux curves on Dupin cyclides.
In section \ref{sc:4} is obtained the differential equation of Darboux curves  in a principal chart and this is applied in subsequent sections.

 In section \ref{sc:6} is defined a natural plane field in the space of spheres and the study of its integrability is carried out.
 In section \ref{sc:7}  a study of Darboux curves near regular ridge points leading to the  zigzag and beak to beak behaviors is carried out.
  In   \ref{sc:8} is considered the study of Darboux curves in general cylinders, cones and  surfaces of revolution, viewed as   canal surfaces.
In section \ref{sc:9} is carried out the global study of Darboux curves on quadrics.

\section{The set of spheres in $ß^3$}\label{sc:1}
\begin{figure}[ht]
\begin{center}
\psfrag{ssx}{\Large $\Sigma$}
\psfrag{lxpur}{\Large $\ell_{\sigma}^{\bot}$}
\psfrag{xx}{\Large $\sigma$}
\psfrag{lx}{\Large $\ell_{\sigma}$ }
\psfrag{lamb}{ \Large $\Lambda$ }
\psfrag{light}{\Large $\LL ight$}
\psfrag{sinf}{\Large ${ß}_{\infty}^3$}
\includegraphics[scale=0.60]{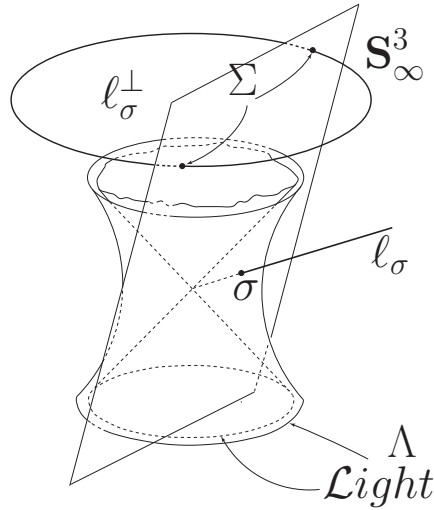}
\caption{${ß}^3_{\infty}$ and the correspondence between points of $\Lambda^4$ and
spheres. \label{fig1}}
\end{center}
\end{figure}

The {\it Lorentz quadratic form} ${\cal L}$ on $\RR^5$ and
the associated {\it Lorentz bilinear form} ${\cal L}(\cdot ,\cdot )$,  are defined by
${\cal L}(x_0,\cdots ,x_4)=-x_0^2+(x_1^2+\cdots +x_{4}^2)$ and
${\cal L}(u,v)=-u_0 v_0+ (u_1 v_1 +\cdots  + u_4v_4 ).$

The Euclidean space $\RR^{5}$ equipped with this pseudo-inner product ${\cal L}$
is called the {\it Lorentz space} and denoted by $\ll^5$.

The isotropy cone
${\cal L}i =\{v\in\RR^{5} \, |\, {\cal L}(v)=0\}$ of ${\cal L}$
is called the {\it light cone}.
Its non-zero vectors  are also called {\it light-like} vectors.
The {\it light cone}  divides the set of vectors $ v\in \ll^5, v\notin \{{\cal L} = 0 \}$ in two classes:

\vskip 1 mm
A vector $v$ in $\RR^{5}$ is called
 {\it space-like} if ${\cal L}(v)>0$ and {\it time-like} if ${\cal L}(v)<0$.

A straight  line is called space-like (or time-like) if it contains a space-like
 (or respectively, time-like) vector.


\vskip 1 mm
The points at infinity of the light cone in the upper half space
$\{x_{0} >0\}$ form a $3$-dimensional sphere. Let it be denoted by
${ß}_{\infty} ^{3}$.
Since it can be considered as the set of lines through the origin
in the light cone,
it is identified with the intersection ${ß}_1 ^{3}$ of
the upper half light cone and the hyperplane $\{x_{0}=1\}$,
which is given by
${ß}_1 ^{3}=\{(x_1, \cdots , x_{4} )\, |\, x_1^2+\cdots
+x_{4}^2-1=0\}$.

To each point $\sigma \in \Lambda^4=\{v\in\RR^{5} \, |\, {\cal L}(v)=1\}$   corresponds a sphere $\Sigma =
\sigma^{\bot}\cap {ß}_{\infty} ^{3}$ or $\Sigma = \sigma^{\bot}\cap
{ß}_1 ^{3}$ (see Figure \ref{fig1}).
Instead of finding the points of $ß^3$ ``at infinity", we can also
consider the section of the lightcone by a space-like affine hyperplane
$H_z$ tangent to the upper sheet of the hyperboloid ${\mathcal H}=\{
\LL
=-1\}$ at a point $z$. This intersection $\LL ight \cap H_z$ inherits
from
the Lorentz metric a metric of constant curvature $1$ (see \cite{H-J},
\cite{La-Wa1}, and Figure \ref{TzH}).
\begin{figure}[htbp]
\begin{center} \includegraphics[scale=0.6]{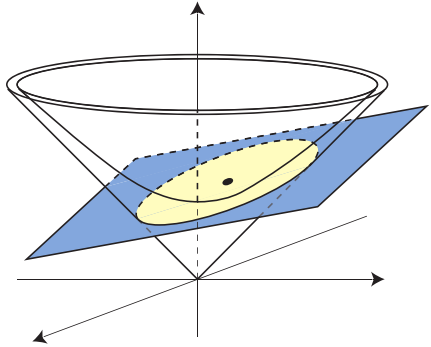}
\caption{A tangent space to $\mathbb{H}^{5}$ cuts the light cone at a unit sphere
\label{TzH}}
\end{center}
\end{figure}
\begin{figure}[ht]
\begin{center}
\psfrag{mm}{$m$}
\psfrag{HH}{$\mathcal H$}
\psfrag{EE}{$\EE^3$}
\psfrag{LL}{$\LL ight$}
\psfrag{SS}{$ß^3$}
\psfrag{vv}{$v$}
\psfrag{Rgn}{$1/k_g\cdot \vect{n}$}
\psfrag{Sig}{$\Sigma$}
\psfrag{Rn}{$1/k\cdot \vect{n}$}
\includegraphics[scale=0.3]{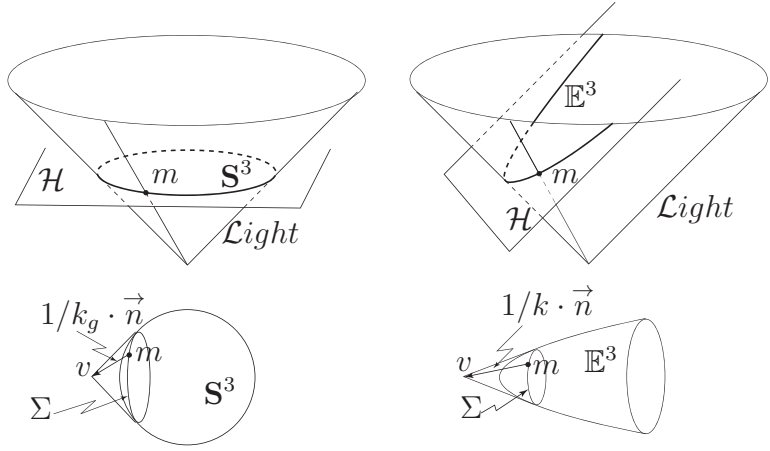}
\caption{Spherical and Euclidean models in the Minkowski space $\ll^5$ (up). The geodesic curvature $k_g$, picture in the affine
hyperplane  $\mathcal H$ (down). 
\label{rad_curv_geod}}
\end{center}
\end{figure}

Notice that the intersection of $\Lambda^4$ with a space-like plane $P$
containing the origin is a circle $\gamma \subset \Lambda^4$ of radius
one
in $P$. The points of this circle correspond to the spheres of a pencil
with base circle. The arc-length of a segment contained in $\gamma$ is equal   to the angle between the spheres corresponding to the extremities of the arc.

It is convenient to have a formula giving the point $\sigma \in
\Lambda^4$
in terms of the Riemannian geometry of the corresponding sphere $\Sigma
\subset ß^3 \subset \LL ight$ and a point $m$ on it. For that we need
to know also the unit vector $\vect{n}$ tangent to $ß^3$ and normal to
$\Sigma$ at $m$ and the geodesic curvature of $\Sigma$, that is the
geodesic curvature $k_g$ of any geodesic circle on $\Sigma$.
\begin{prop}\label{prop_sigma=k_g m+n}
The point $\sigma\in \Lambda^4$ corresponding to the sphere $\Sigma
\subset ß^3 \subset \LL ight$ is given by
\begin{equation}\label{sig_sig}
\sigma = k_g m+\vect n.
\end{equation}
\end{prop}
\begin{rema}
A similar proposition can be stated for spheres in the Euclidean space
$\EE^3$ seen as a section of the light cone by an affine hyperplane parallel to an hyperplane tangent to the light cone.
\end{rema}

\medskip
The proof of Proposition \ref{prop_sigma=k_g m+n} can be found in
\cite{H-J} and \cite{La_Oh}.
The idea of the proof is shown on Figure \ref{rad_curv_geod}:
Let $\mathcal H$ be the affine hyperplane such that $ß^3 = \LL ight
\cap {\mathcal H}$,
let $P$ be the hyperplane such that $\Sigma = ß^3 \cap P$; the vertex
of the cone, contained in $\mathcal H$, tangent to $ß^3 $ along
$\Sigma$ is a point of the line $P^{\bot}$ which contains the point
$\sigma\in \Lambda^4$.
\begin{figure}[ht]
 \begin{center}
\psfrag{siga}{$\sigma_{\alpha}$}
\psfrag{sig1}{$\sigma_{1}$}
\psfrag{sig2}{$\sigma_{2}$}
\psfrag{tang}[c][c][0.6][0]{\ $\begin{array}{c}\text{lightrays in}\ \Lambda\\ \text{corresponding to a}\\ \text{pencil of tangent spheres}\end{array}$}
\psfrag{atm}[c][c][0.6][0]{$\begin{array}{c}\text{lightray in}\ \Lambda\\ \text{corresponding to }\\ \text{oriented spheres }\\ \text{tangent to }$M$ \text{ at }$m$\end{array}$}
\includegraphics[scale=0.6]{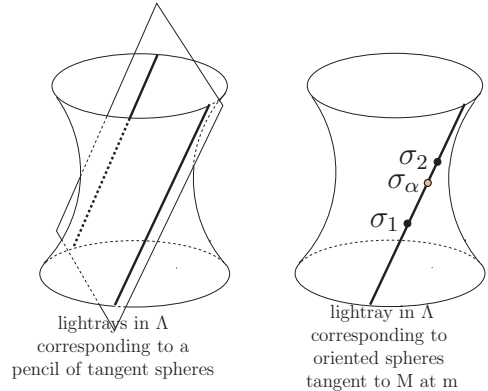}
\caption{Tangent spheres \label{tang_spheres}}
\end{center}
\end{figure}
From Proposition \ref{prop_sigma=k_g m+n}, we see that the points in $\Lambda^4$ corresponding to a pencil of spheres tangent to a surface $M$ at a point $m$ form two parallel light-rays (one for each choice of normal vector $n$). Let us now chose the normal vector $n$, and consider the spheres $\Sigma_k$ associated to the point $\sigma_k = km+n$.
\begin{figure}[ht]
\begin{center}
\psfrag{aa}[c][c][0.6][0]{Saddle contact}
 \psfrag{bb}[c][c][0.6][0]{$\begin{array}{c}\text{Contact with an}\\
 \text{osculating sphere}\end{array}$}
 \psfrag{cc}[c][c][0.6][0]{Center contact}
\includegraphics[scale=0.6]{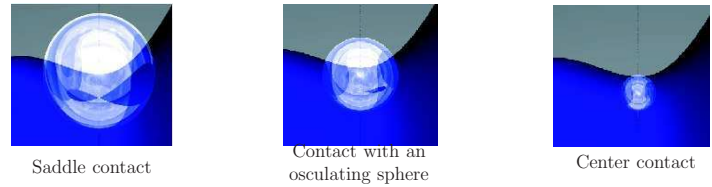}
\caption{Possible contacts of a sphere and a surface \label{contact}}
\end{center}
\end{figure}
%
All the spheres $\Sigma_k$, but for at most two, have either a center contact or a saddle contact with $M$ (see Figure \ref{contact}).

The exceptional spheres correspond to the principal curvatures of $M$ at $m$, they are called osculating spheres.

When $k\notin [k_1 , k_2]$ the intersection of $\Sigma_k$ and $M$ near the origin reduces to a point, the origin.

When $k\in (k_1 , k_2)$ the intersection of $\Sigma_k$ and $M$ near the origin consists of two curves intersecting transversely at $m$.

When $k=k_1 $ or $k_2 $ the intersection of $\Sigma_k$ and $M$ near the origin is  a singular curve, in general of cuspidal type,  at $m$.

In fact one can prove the following

\begin{prop}\label{angle_intersec}
Let $k= k_1 cos^2 \alpha  + k_2 sin^2 \alpha$. Then the angle of the  tangents at $m$ to $\Sigma_k \cap M$ with the principal direction corresponding to $k_1$ is $±\alpha$.
\end{prop}

In general, the intersection of an osculating sphere with $M$ admits a cuspidal point at $m$, the tangent to the cusp is then the principal direction associated to the curvature $k_i$. In any case, when $k\in [k_1,k_2]$ goes to $k_i$, the two tangent directions at $m$ to $\Sigma_k \cap M$ converge to the principal direction associated to the curvature $k_i$. From that we can see that lines of curvature as osculating spheres are conformally defined.
\section{Foliations making a constant angle with respect to Principal Foliations}\label{sc:5}
Consider a surface $ M$ with principal foliations ${\mathcal P}_1$ and ${\mathcal P}_2$ and umbilic set $\mathcal U$. The triple  ${\mathcal P}=({\mathcal P}_1, {\mathcal P}_2,  \mathcal U)$ will be referred as the {\ em principal configuration} of the surface.

\begin{defi}\label{fol_alpha}
For each angle $\alpha\in (-\pi/2,\pi/2)$ we can consider the  foliations  $\fol_{\alpha}^{+} $ and  $\fol_{\alpha}^{-} $ such the leaves of this foliation are the curves making a constant angle $±\alpha$ with the leaves of the principal foliation ${\mathcal P}_1$. We will write    $\fol_{\alpha}= \{ \fol^+_{\alpha}, \fol^-_{\alpha}\}. $

In other words, the normal curvature of a leaf  of $\fol_{\alpha}$ is precisely $k_n(\alpha)=k_1\cos^2\alpha+k_2\sin^2\alpha$.
\end{defi}
\subsection{Foliations $\fol_{\alpha}$ on Dupin cyclides}
%
 \begin{figure}[htb]
\begin{center}
\includegraphics[scale=0.4]{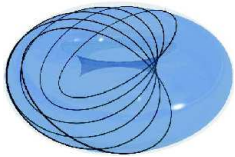}
\caption{Foliation of a torus of revolution by Villarceau circles\label{Villarceau}}
\end{center}
\end{figure} 
Dupin cyclides are very special: they are surfaces which are in two different ways envelopes of one-parameter families of spheres (see \cite{Da3}). This implies that the corresponding curves are circles or hyperbolas in $\Lambda^4$, intersection of $\Lambda^4$ with an affine plane (see \cite{La-Wa1}).

There are three types of Dupin cyclides. One can chose a nice representant of each class: 

- A) The boundary of a tubular neighbourhood of a geodesic of $ß^3$

- B) A cylinder of revolution in $\RR^3$

- C) A cone of revolution in $\RR^3$.

Then, in cases A) and B) the foliations $\fol_{\alpha}^+$ and $\fol_{\alpha}^-$ are totally geodesic foliations.

In the case B), four foliations are foliation by circles: the two foliation by characteristic circles $\fol_0$ and $\fol_{\pi/2}$, and two others: the foliations by Villarceau circles.

In the case C), we have to develop the cone on a plane --this is a local isometry out of the origin-- to see that a foliation of the plane by curves making a constant angle with rays is a foliation by logarithmic spirals. The picture on the cone is obtained by rolling the plane foliation back on the cone.

\subsection{Foliations $\fol_{\alpha}$ on quadrics}
\begin{prop}\label{prop:54ma}
Consider an ellipsoid $\mathbb E_{a,b,c}=\{(x,y,z):\; \frac{x^2}{a }+\frac{y^2}{b }+\frac{z^2}{c }=1\}$  with    $a>b>c>0$. Then $\mathbb E_{a,b,c}$ have four
umbilic points located in the plane of symmetry  orthogonal to middle axis;   they are of the
 type $D_1$,   i.e., a singularity of index $1/2$ and having one separatrix  for the principal curvature lines. For all $\alpha$ the singularities of $\fol_{\alpha}$  are the four umbilic points and this configuration
 is  topo­lo­gically equivalent
 to the principal configuration $\mathcal P$ of the ellipsoid near the umbilic points which are of type D$_1$  \end{prop}

\begin{proof} 
Consider the parametrization of the ellipsoid in a neighborhood of the umbilic point
$p_0= (u_0,0,v_0)= (  \sqrt{\frac{a(a -b )}{a -c }},\,0,\,
 \sqrt{\frac{c(c -b )}{c -a }}\;).$
{\small
 $$\alpha(u,v)= p_0 +u E_1 +v E_2+\frac 12[ \sqrt{ac}(u^2+v^2)+\frac{\sqrt{c(b-c)(a-b)}}{b^3}(u^3+uv^2)+h.o.t]E_3
 +h.o.t.$$
}
Here   $\{E_1,E_2,E_3\}$, $E_2=(0,1,0)$, is a positive orthonormal base and the ellipsoid is  oriented by $E_3=- \nabla H(p_0)/|\nabla H(p_0).$

In  a neighborhood of the umbilic point $(0,0)$
 the differential equation
  of the foliation  $\fol_{\alpha}$ in the chart $(u,v)$ is given by:

  $A(u,v)dv^2+B(u,v)dudv+C(u,v)du^2=0$, where

  $$\aligned A(u,v)=&-u-\cos 2\alpha \sqrt{u^2+v^2+R_3(u,v)}+ A_2(u,v)\\
  B(u,v)=&2v+B_2(u,v)\\
  C(u,v)=& u-\cos 2\alpha \sqrt{u^2+v^2+R_3(u,v)}+ C_2(u,v)\endaligned$$
 Here  $A_2=O(r^2),\; C_2=O(r^2)$ and $R_3(u,v)=O(r^3), \; r=\sqrt{u^2+v^2}$.

The above implicit equation   has,  real separatrices with limit direction given by $±2\alpha$
 and the behavior of the integral curves  near $0$ is the same of an umbilic point of type $D_1$.

 In fact, consider the blowing-up $u=r\cos\alpha,\;v=r\sin\alpha$.

 The differential equation of  $\fol_{\alpha}$  in the new variables is given by:
 $$\aligned (&\cos 2\alpha-\cos\alpha+rR_1(r,\alpha) )dr^2+
 r (2 \sin\alpha+rR_2(r,\alpha))dr d\alpha\\
+&r^2(\cos 2\alpha 
+ \cos\alpha+r R_3(r,\alpha))d\alpha^2=0.\endaligned$$

 The two singular points are given by $(2\alpha,0)$ and $(-2\alpha,0)$. Direct analysis shows that both singular points are hyperbolic saddles of the adapted vector fields to the implicit equation near these singularities. The blowing-down of the saddle separatrices are the umbilic separatrices of $\fol_{\alpha}^+ $ and $\fol_{\alpha}^- $. See Fig. \ref{fig:folrodelip}.

  \begin{figure}[ht]
 \psfrag{p1}{$p_1$}
    \psfrag{p2}{$p_2$}
  \psfrag{p3}{$p_3$}
 \psfrag{p4}{$p_4$}
\begin{center}
\includegraphics[scale=0.50]{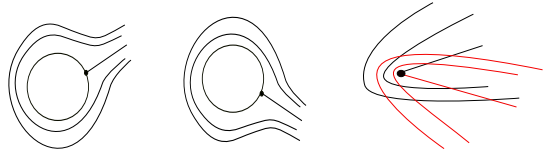}
\caption{Resolution of the foliations $\fol_{\alpha}^+$ and $\fol_{\alpha}^-$.  \label{fig:folrodelip} }
\end{center}
\end{figure}

Therefore it follows that
 the pair of  foliations $\fol_{\alpha}^+$  and $\fol_{\alpha}^-$   near an umbilic point of
 the ellipsoid with three distinct axes is  topologically equivalent to the configuration
of principal lines near a Darbouxian umbilic point $D_1$.
\end{proof}

\begin{rema} The study of principal lines was first considered by Monge, see \cite{Mo5} and \cite{Mo6}.
 Near umbilic points the behavior of principal lines on real analytic surfaces was established by Darboux, \cite{Da2}. See also \cite{Gu}, \cite{gs1}, \cite{gascoloquio} and references therein.
\end{rema}

\begin{prop}\label{prop:55ma}
Consider an  \index{ellipsoid} ellipsoid $\mathbb E_{a,b,c} $ with   $a>b>c>0$. On the ellipse $\Sigma_{xz}\subset
\mathbb E_{a,b,c}$, containing the four  umbilic points,$\;p_i$, ($i=1,\cdots,4 \;$) counterclockwise
oriented, denote by $s_1(\alpha)=2\int_{c}^{b}\sin\alpha[\frac{\sqrt{u}}{(-H(u))}]du$  {\rm (} resp. $ s_2(\alpha)=2\int_{b}^{a}\cos\alpha[\frac{\sqrt{v}}{H(v)}]dv$ {\rm )}  a distance between the adjacent umbilic points
$p_1$ and $p_4$ {\rm (} resp. $p_1$ and $p_2$ {\rm )}. Define $\rho(\alpha)=\frac{s_2(\alpha)}{s_1(\alpha)}$.

Then if $\rho\in\mathbb R\setminus \mathbb Q$ {\rm (} resp. $\rho\in \mathbb Q$)
 all the leaves of $\fol_{\alpha}$ are recurrent {\rm (} resp. all, with the exception of the   umbilic separatrices, are closed{\rm )}. See Fig. \ref{fig:51ma}. \end{prop}

 \begin{figure}[ht]
 \psfrag{p1}{$p_1$}
    \psfrag{p2}{$p_2$}
  \psfrag{p3}{$p_3$}
 \psfrag{p4}{$p_4$}
\begin{center}
 \includegraphics[scale=0.60]{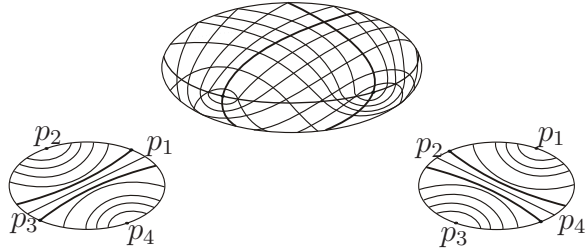}
\caption{Foliations ${\mathcal R}_i$ of  the ellipsoid $\mathbb E_{a,b,c}$ \label{fig:51ma} }
\end{center}
\end{figure}

\begin{proof}
 The ellipsoid $\mathbb E_{a,b,c}$ belongs to  the tri­ple or­tho­gonal system of surfaces defined by the
one pa­ra­meter fa­mi­ly of quadrics,
$\frac{x^2}{a -\lambda} +\frac{y^2}{b -\lambda}+\frac{z^2}{c -\lambda}=1$ with $a>b>c> 0$, see
also \cite{spivak} and \cite{struik}.
The  following parametrization $\alpha(u,v) =   (x(u,v),y(u,v),z(u,v))$ of $\mathbb E_{a,b,c}$, where

{\small
\begin{equation}\label{eq:porto}
\alpha(u,v)=\left(
 ±\sqrt{\frac{a (u-a )(v-a )}{(b -a )(c -a )}},  
 ±\sqrt{\frac{b (u-b )(v-b )}{(b -a )(b -c )}},  
  ±\sqrt{\frac{c (u-c )(v-c )}{(c -a )(c -b )}}\right),  \end{equation} }
\noindent defines  the principal coor­dina­tes $(u,v)$ on $\mathbb E_{a,b,c}$, with $u\in
(b,a)$ and $v\in (c,b)$.

The first fundamental form of $\mathbb E_{a,b,c}$  is given by:

\begin{equation} \label{eq:Ie} I=ds^2=Edu^2+Gdv^2= \frac{ (v-u)u}{4H(u)}du^2+
 \frac{ (u-v)v}{4H(v)}dv^2 \end{equation}

The second fundamental form with respect to the normal
$N=-(\alpha_u\wedge \alpha_v)/|\alpha_u\wedge \alpha_v|$ is given
by

\begin{equation} \aligned \label{eq:IIe} II=& edu^2+gdv^2=\frac{(v-u)}{4 H(u)}\sqrt{\frac{abc}{uv}}
du^2+\frac{(u-v)}{4 H(v)}\sqrt{\frac{abc}{uv}}dv^2\\
H(x)=& (x-a)(x-b)(x-c)\endaligned
\end{equation}


Therefore the principal curvatures are given by:

$$k_1=\frac{e}E=  \frac 1u \sqrt{\frac{abc}{uv}},\;\; k_2=\frac gG=\frac 1v
\sqrt{\frac{abc}{uv}}.$$

The four umbilic points are given by:
{\small
 { $$(±x_0,0,±z_0)= (± \sqrt{\frac{a(a -b )}{a -c }},\,0,\, ±
 \sqrt{\frac{c(c -b )}{c -a }}\;).$$}}
\vskip .3cm


The differential  equation of the foliation  $\fol_{\alpha}$ in the principal chart $(u,v)$ is given by
$$\aligned &H(u)v\cos^2\alpha \;dv^2 +H(v)u\sin^2\alpha\; du^2= 0 \;\;  \Leftrightarrow  \\
&\frac{v}{H(v)} \cos^2\alpha \;dv^2+\sin^2\alpha\; \frac{u}{H(u)}du^2=0.\endaligned$$
Define $d\sigma_1=  \sin\alpha \sqrt{\frac{u}{(-H(u))}}du$ and $d\sigma_2= \cos\alpha \sqrt{\frac{v}{ H(v)}}du$

Therefore the differential equation of $\fol_{\alpha}$  is equivalent to $d\sigma_2^2-d\sigma_1^2=0$ in the rectangle $[0,s_1(\alpha)]×[0,s_2(\alpha)]$.

 On the ellipse $\Sigma=\{(x,y,z) |\{\frac{ x^2}{a }+\frac{z^2}{c }=1,\; y=0\}$
define the   distance between the umbilic points $p_1=(x_0,0,z_0)$ and $p_4=(x_0,0,-z_0)$
 by
$s_1(\alpha)=2\int_{c}^{b}\sin\alpha[\frac{\sqrt{u}}{(-H(u))}]du$ and that  between the umbilic points
$p_1=(x_0,0,z_0)$ and $p_2=(-x_0,0,z_0)$
 is  given by
$s_2(\alpha)=2\int_{b}^{a}[\cos\alpha\frac{\sqrt{v}}{H(v)}]dv$.

The ellipse $\Sigma$ is the union of four umbilic points and   four principal umbilical
separatrices for the principal
foliations.  So $\Sigma\backslash\{p_1,p_2,p_3,p_4\}\;$ is a transversal section of the foliation $\fol_{\alpha}$.

  Therefore near the  umbilic point $p_1$  the   foliation, say $\fol^1_{\alpha}$,  with
umbilic separatrix contained in  the region $\{y> 0\}$ define a the return map $\sigma_+:\Sigma\to
\Sigma$ which is an isometry,  reverting the orientation, with $\sigma_+(p_1)=p_1$. This follows
because in the principal chart  $(u,v)$ this return  map is defined by $\sigma_+:\{u=b\}\to
\{v=b\}$ which satisfies the  differential equation $\frac{ds_2}{ds_1}=-1$. By analytic
continuation it results that $\sigma_+$  is a isometry reverting orientation with two fixed points
$\{p_1,\;p_3\}$. The geometric reflection $\sigma_-$,  defined in the region $y<  0$ have the two
umbilic
$\{p_2,\;p_4\}$ as fixed points.
So the Poincaré return map $\pi_1:\Sigma\to \Sigma$ (composition of two isometries $\sigma_+$ and
$\sigma_-$) is a rotation with rotation number given by ${s_2}(\alpha)/{s_1}(\alpha)$.

Analogously for   $\fol^2_{\alpha}$ with the Poin­caré re­turn
map given by $\pi_2=\tau_+\circ \tau_-$ where $\tau_+$ and $\tau_-$ are two
isometries having respectively   $\{p_2,p_4\}$ and $\{p_1,p_3\}$ as
fixed points.
\end{proof}

\begin{rema} The special case $\alpha=\pi/4$ was studied in \cite{gasmedia}. A more general framework of implicit differential equations, unifying various families of geometric curves    was studied in \cite{gsmg}. See also \cite{gascoloquio}.
\end{rema}

\begin{prop} \label{caract_Dupin_via_Darboux} In any surface, free of umbilic points, the leaves of   $\fol_{\alpha}^+$ and $\fol_{\alpha}^-$  are Darboux curves if and only if the surface  is conformal to  a Dupin cyclide.
\end{prop}

\begin{proof} From the differential equation of Darboux curves, see  equation \eqref{eq:dpc} in Section \ref{sc:4}, it follows that:
$$3(k_1-k_2)\sin\alpha\cos\alpha\;\frac{d\alpha}{ds} =  \frac{1}{\sqrt{E}} \frac{\partial k_1}{\partial u}
 \cos^3\alpha+ \frac{1}{\sqrt{G}}\frac{\partial k_2}{\partial v}
\sin^3\alpha.$$
So all  Darboux curves are leaves of $\fol_{\alpha}^± $  if and only if  $ \frac{\partial k_1}{\partial u}(u,v)
=  \frac{\partial k_2}{\partial v}(u,v)=0$. By Proposition 
 \ref{prop02canal} in Section \ref{sc:8} this is exactly the
 condition that characterizes the Dupin cyclides. 
\end{proof}
\section{Spheres tangent to a  surface along a curve}\label{sc:2}
Above each point $m\in M$ which is not an umbilic, the spheres
 tangent to the surface $M$ having a saddle contact with $M$ form
 an interval of boundary the two osculating spheres at $m$. Let 
us first define the 3-dimensional subset $V(M)\subset \Lambda^4$ 
as the union of the spheres having a saddle contact with the surface $M$ and the osculating spheres.

Therefore $V(M)$ is an interval fiber-bundle $\pi V(m)\rightarrow M$ over $M$;
 the boundary of $V(M)$ is  the  surface ${\mathcal O}$ of spheres osculating $M$
 (this surface has two connected components when the surface $M$ has
 no umbilical point). $V(M)$ is a 3-manifold with boundary a surface
 ${\mathcal O}$ made of spheres osculating $M$ (this surface has 
two connected components when the surface $M$ has no umbilical point).
 $V(M)$  inherits from $\Lambda^4$ a semi-Riemannian metric.
 At each point, the kernel direction is the direction of the light ray through the point.

Let us consider now a curve
$C\subset M^2$.

The restriction of $V(M)$ to $C$  form a two-dimensional surface
 in $\Lambda^4$ which is a light-ray interval bundle $V(C)$ over
 $C$ out of the umbilical points of $M$ which may belong to $C$. 
 From $V(M)$, we get on $V(C)$ an induced semi-Riemannian metric.

In this text we will use {\Large $'$} for derivatives with respect 
to parametrization of curves contained in $\RR^3$ or $ß^3$, (often 
 the parameter is an  arc length), and $\Pt{}$ for derivatives with
 respect to a parametrization of a curve  in $\Lambda^4$ (often 
 the parameter is an  arc length, but now for the metric induced from the Lorentz ``metric").

One section of this bundle has the property that at every point $m$ where
 the curve is not tangent to a principal direction, the point
 $\sigma_{C'}(m)$  correspond to a sphere $\Sigma(m)$ such that
 one branch of the intersection $\Sigma_{C'}(m)\cap M$ is tangent
 to $C$ at $m$. If at $m$ the curve $C$ is tangent to a principal
 direction we take the corresponding osculating sphere as $\Sigma_{C'}(m)$.

We call  this section $Cansec(C)$ of $V(C)$ the canonical section, and
 the corresponding family of spheres the canonical family along $C$; 
the envelope $CanCan(C)$ of the spheres of this family is called canonical
 canal corresponding to $C\subset M$.

\begin{prop}\label{Cansec} Consider the notation above. Then the following holds:

i)  The section  $Cansec(C)$  satisfies ${\cal L}(\overrightarrow{k_g})>0$,
 and moreover $\overrightarrow{k_g} \in T_m V(M)$.

ii)  The section $Cansec(C)$ is a geodesic in $V(C)$. It is of minimal
 length among sections of $V(C)$.
 Moreover it is the unique section on $V(C)$ which is of critical  length.

iii)  One of the cuspidal edges of the envelope  $CanCan(C)$ is $C$.
\end{prop}

\begin{proof}
The condition defining the sphere $\Sigma_{C'}(m)$ implies that the order of
 contact  of  $C$ and $\Sigma_{C'}(m)$ is at least $2$,
 one more that the order of  contact   of $\Sigma_{C'}(m)\cap M)$ and $C$, which is at least one.
To verify this property, notice that, in terms of the arc-length  of a branch of
 $\Sigma_{C'}(m)\cap M)$, or equivalently of the arc-length $s$ on $C$,
 the angle of $\Sigma(m)$ with $M$ along $\Sigma_{C'}(m)\cap M)$ is of
 the order of $s$, if not smaller. The distance to $C$ is of the order 
of $s^2$, if not smaller. Therefore the order of the distance of a point
 of $C$ to $\Sigma_{C'}(m)$ is of order $s^3$ if not smaller. This means 
that $C$ and the sphere $\Sigma_{C'}(m)$ have contact of order at least 2 at $m$.

\begin{defi}\label{def:da}
A curve  on $M$ such that at each point its osculating sphere is
 tangent to $M$ will be called  {\em Darboux curves}.
\end{defi}

Recall a lemma from \cite{laso}

\begin{lemm}\label{contacte}
A curve $\Gamma(t)=\spa(\gamma(t))$ has contact of order $\geq k$
with a sphere $\Sigma$ corresponding to $\sigma$ iff
\[
\sigma\bot\spa(\gamma(t),\Pt{\gamma(t)},\ldots,\gamma^{(k)}(t))
\]
\end{lemm}
 
\begin{proof}
The sphere $\Sigma$ is the zero level of the function
$f(x)=\langle x| \sigma\rangle$. Then the contact of $\Gamma(t)$
and $\Sigma$ has the order of the zero of
$(f\circ\gamma)(t)=\langle \gamma(t)| \sigma\rangle$.
\end{proof}
Therefore, the sphere $Cansec(m)$ is orthogonal to $m$, $\Pt{m}$ and $\Ppt{m}$.
 This ends the proof of Proposition \ref{Cansec}.
 \end{proof}
\vskip 3 mm
Differentiating  the relation $\langle \sigma |\Pt{m}\rangle = 0$ and
 using the relation $\langle \sigma |\Ppt{m}\rangle = 0$ we get
 $\langle \Pt{\sigma} |\Pt{m}\rangle = 0$. Differentiating the
 relation $\langle \Pt{\sigma} |{m}\rangle = 0$, and using the 
relation $\langle \Pt{\sigma} |\Pt{m}\rangle = 0$, we get the
 relation $\langle \Ppt{\sigma} |{m}\rangle = 0$.

We can use as parameter the arc-length of the space-like curve
 $\sigma = Cansec$. Its geodesic curvature vector, orthogonal to $\hat{m}$ is therefore
orthogonal to $\hat{m}\oplus \RR\, \Pt{\sigma}$, proving that
 $\{ \sigma \}$ is a geodesic in $V(C)$.

We notice that $\Ppt{\sigma}$ is tangent to $T_m V(M)$, as
$T_m (V(M)) = T_m({\mathcal L}ight) = \hat{m}^{\bot}$.

Now we will use the Darboux frame $T, N_1,n , m$ of the curve 
$C\subset M\subset ß^3 \subset \LL ight$, where $N_1$ is the unit 
 vector tangent to $M$ normal to $C$ compatible with the orientation of $M$.
Using again the formula $\sigma = k_n \, m +n$, where $k_n$ is the normal
 curvature in the direction tangent to $C$ at $m$, we see that, when $\sigma$ is the  section $Cansec(C)$
\begin{equation}\label{geodesic_torsion_can}
|\sigma'|= |k_n'\, m + k_n\, m' +n'= |k_n'\, m - \tau_g N_1|= |\tau_g|
\end{equation}
\noindent where the {\it geodesic torsion} $\tau_g$ is defined by the following formula
\begin{equation} \label{geodesic_torsion}
\Scp{(\nabla_T n)}{N_1}\, =\,- \tau_g.
\end{equation}
Observe that our formula (\ref{geodesic_torsion_can}) gives an
 interpretation of the geodesic torsion of a curve $C\subset M$ as
 the  rotation speed of the canonical  family  of spheres tangent to $M$ along $C$.

In order to compute the geodesic curvature vector of the section $Cansec(C)$, we
 need to use its parametrization by arc-length in $\Lambda^4$.
Then $\Pt{\sigma}= \sigma'\frac{1}{\tau_g}= \frac{1}{\tau_g}(k_n'\, m - \tau_g N_1)$.
 Differentiating  once more, we get

$$\Ppt{\sigma}= -\frac{\Pt{\tau_g}}{\tau_g^3}(k_n'\, m - \tau_g N_1) 
+ \frac{1}{\tau_g^2}[k_n''m + k_n' T - \tau_g'N_1 - \tau_g (-k_g T + \tau_g n) ],$$
\noindent  notice that  $k_g$ is the geodesic curvature of $C \subset M$,
 while $\vect{k_g}$ is the geodesic curvature vector  of the curve $Cansec \subset \Lambda^4$.

Simplifying, we get
$$\Ppt{\sigma} =  \phi(s) m + \frac{1}{\tau_g^2} [(k_n' +\tau_g  k_g)T -\tau_g^2 n]$$
As $\vect{k_g} = \Ppt{\sigma} + \sigma$ we get
\begin{equation} \label{geodesic_cansec}
\vect{k_g} = \psi(s)m +  \frac{1}{\tau_g^2} (k_n' +\tau_g k_g)T
\end{equation}


As the vector $\vect{k_g}$ is orthogonal to $m$ and to $\Pt{\sigma}= \frac{1}{\tau_g}(k_n'\, m - \tau_g N_1)$, we see that the curve $Cansec$ is a geodesic on $V(C)$.

Otherwise, when $\sigma= k \, m+ n$, we see that spheres tangent to a surface
  along a curve form a space-like curve in $\Lambda^4$; explicitly we get
\begin{equation}\label{length_section_cansec}
 |\Pt{\sigma}|= |\frac{1}{\tau_g}[(k-k_n)'m\,+\,(k-k_n) T\, +\, \tau_g N_1]|,
\end{equation}
\noindent giving, as $m$, $T$ and $N$ are mutually orthogonal, a proof of 
the fact that the  section $Cansec(C)$ has minimal length among sections, as the
 three vectors $m,T,N$ are mutually orthogonal in $\ll^5$. Formula 
\ref{length_section_cansec} shows also that no other section of $V(C)$ is of critical length.

\begin{rema}
 The characteristic circle of the envelope is the intersection
 of $span(\sigma, \Pt{\sigma})^{\bot}$ and the sphere $ß^3 \subset \RR^4$; as the vector $T$ is orthogonal 
to $m$, $n$ and $N_1$, and therefore to $\sigma$ and $\sigma'$
 (and also $\Pt{\sigma}$), the characteristic circle  is tangent at $m$ to $C$.
\end{rema}
\vskip .2cm

Recall  that a {\it drill}  (\cite{Tho}, \cite{laso}) is a curve
 in the space of spheres the geodesic curvature of which is light-like  at each point.
Generically, points of a drill are osculating spheres to the curve 
$C\in  \bs^3$ defined by the geodesic acceleration vector $\vect{k_g}$ of the drill.

We see that if we can find drills in $V(m)$ we find geodesics.
 In fact we find that way almost all of them.

Formula \ref{geodesic_cansec} implies, as $m$ is the normal direction
 to $V(M)$ at a point $\sigma = km+n$, the following theorem.
\begin{theo}\label{geodesic_Darboux_curves}
The curve $C$ is a Darboux curve  if and only if the section 
 $Cansec(C) \subset V(M)$ is a geodesic in $V(M)$. This happens
 if and only if the curve $C\subset M$ satisfies
the equation:
 \begin{equation}\label{eq:eqdd} k_n' +\tau_g k_g =0.\end{equation}
\end{theo}
\begin{rema}
The light rays of $V(M)$ are also geodesics. Segments of geodesics 
of $V(M)$ which are not tangent to light rays define an arc of 
curve $C\subset M$, and therefore, when $\vect{k_g}$ is light-like,
 $C$ is a piece of Darboux curve on $M$.

Also when $C$ is a Darboux curve  the  only   cuspidal edge  of 
the envelope  $CanCan(C)$ is $C$. See \cite{Tho}.

\end{rema}

\begin{defi}\label{meilleure_sphere}
We denote by $\Sigma_{m,\ell}$ the sphere tangent to the surface $M$
 at the point $m$ such than one branch of the intersection $\Sigma_{\Pt c} \cap M$
 is tangent to the direction $\ell$ ({\it canonical} sphere associated
 to the direction $\ell$).  We will also use the notation $\Sigma_{m,v}$ when 
the direction $\ell$ is generated by a non-zero vector $v$.
\end{defi}

Given a curve $C\subset M$ such that the tangent vector to $C$ at $c(t)$
 is contained in $\ell$ , $\Sigma_{c(t),\Pt c(t)}$ is, among the spheres
 tangent to $M$ at $m$, the one which have the best contact at $m$ with the curve $C$.

\begin{rema}
$\tau_g ds$ is the differential of the  rotation of the sphere 
$\Sigma_{c(t),\Pt c(t)}$ along the curve $C$.
\end{rema}

\begin{demo}
The sphere $\Sigma_{c(t),\Pt c(t)}$ has a tangent movement which
 is a rotation of ``axis" the characteristic circle of the family,
 which is tangent to $C$. It is tangent to the tangent plane. We see
 that changing the sphere tangent to the surface along $C$ changes the 
``pitch" term but not the ``roll" term.
\end{demo}

\begin{prop}
Let $C$ be a curve contained in the surface $M\subset ß^3$.
 Let $\gamma_C$ be the curve in
$\Lambda^4$ obtained considering at each point $m\in C$ the
 sphere $\Sigma_{m,\ell(m)}$, where $\ell(m)$ is the direction 
tangent at $m$ to $C$. Suppose that the curve $C$ is nowhere 
tangent to a principal direction of curvature. Then the curve $\gamma_C$ is
 space-like with at every point a space-like geodesic acceleration. Moreover
 the curve $C$ is one fold of the singular locus of the canal surface defined by $\gamma_C$.
\end{prop}

\begin{demo}
The condition defining the canonical sphere $\Sigma_{c(t),\Pt c(t)}$ guarantees
 that the characteristic circle $CC(t)$ is tangent to $C$, which is therefore the singular curve of the canal.
\end{demo}
\begin{prop}
The condition of the equation in the theorem (\ref{geodesic_Darboux_curves})
 is equivalent to the fact that one branch of the intersection of the sphere
 $\Sigma_{c(t),\Pt c}$ has at the point the same geodesic curvature as the curve $C$.
\end{prop}

\begin{demo}
As the sphere $\Sigma_{c(t),\Pt c}$ is osculating the curve $C$ when it
 is a Darboux curve, the contact order with the curve should be at least $3$.
\end{demo}

When the direction $\ell$ defined by $\Pt c(t)$ is not a principal direction, 
there is a unique sphere tangent to $M$ which has contact with $C$ of best 
order, the other having contact of order one with $C$. There is only one 
such sphere, in the pencil containing the osculating circle to $C$ which
 is tangent to $M$. This last sphere should therefore be
$\Sigma_{c(t),\Pt c(t)}$

\begin{rema}
The previous considerations imply that the geodesic curvature of the
 branch of $S_{\alpha}\cap M$ tangent  to the   Darboux curve is equal to:
$k_g = \frac{-k'_n}{\tau_g}$.
\end{rema}


%
\section{Darboux curves in cyclides: a geometric viewpoint}\label{sc:3}
   We can describe $V(M)$ when the surface  $M$ is a regular Dupin 
cyclide (see Figure \ref{VMDupin}). It is the wedge of the two circles
 formed by the osculating spheres of the regular cyclide.

\begin{figure}[ht]
 \begin{center}
\psfrag{so1}[h][r][0.8][0]{\hskip 1cm Osculating Spheres  ($k_1$)}
\psfrag{so2}[h][r][0.8][0]{\hskip 1cm Osculating Spheres  ($k_2$)}
\psfrag{sm}[h][r][0.8][0]{\hskip 1cm mean spheres}
\includegraphics[scale=0.4]{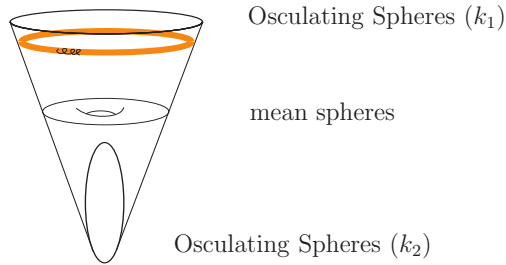}
\caption{Darboux curves in  $V(M)$ when  $M$ is a regular  Dupin cyclide \label{VMDupin}}
\end{center}
 \end{figure}
%

%
Consider a surface $ M$ with principal foliations ${\mathcal P}_1$
 and ${\mathcal P}_2$ and umbilic set $\mathcal U$. The triple 
 ${\mathcal P}=({\mathcal P}_1, {\mathcal P}_2,  \mathcal U)$ will
 be referred as the {\ em principal configuration} of the surface.

\begin{defi}
For each angle $\alpha\in (-\pi/2,\pi/2)$ we can consider the  
foliations  $\fol_{\alpha}^{+} $ and  $\fol_{\alpha}^{-} $ such 
the leaves of this foliation are the curves making a constant
 angle $±\alpha$ with the leaves of the principal foliation ${\mathcal P}_1$.

In other words, the normal curvature of a leaf  of $\fol_{\alpha}$ is precisely $k_n(\alpha)=k_1\cos^2\alpha+k_2\sin^2\alpha$.
\end{defi}

Let us, for later use, define the surfaces $M_{\alpha} \subset V(M)$ as
 the set of spheres tangent to $M$ at a point $m$ having curvature 
$k_{\alpha}=k_n(\alpha)=k_1 \cos^2\alpha  +k_2 \sin^2\alpha$. Each 
surface $M_{\alpha}$ is foliated by the lifts of the curves of
 $\fol_{\alpha}$; we call this new foliation $\tilde{\fol}_{\alpha}$.
 These later foliation form a foliation of $V(M)$ that we call $\tilde{\fol}$.
\begin{prop}\label{prop:dcyclides}
The  Darboux curves in  Dupin cyclides are the leaves of the foliations
 $\fol_{\alpha}^+$ and  $\fol_{\alpha}^-$.
\end{prop}

The proof of Proposition \ref{prop:dcyclides} is given in   subsections \ref{sus:1} and  \ref{sus:2}.
\subsubsection{Cylinders in $\RR^3$}\label{sus:1}
This is the easiest case to visualize. Cylinders are models for cyclides having exactly one singular point.

The definition of a Darboux curve requires that the osculating
 sphere to the helix  $\cal H$ drawn on the revolution cylinder
be tangent to it. This comes from the fact that the helix is
 invariant by the rotation  $R$ of angle  $\pi$ and axis equal to the
 principal curvature vector of the helix which is a vector
 orthogonal to the cylinder and going along its axis of 
revolution. As the osculating sphere should also be invariant
 by this rotation its diameter in contained in the line normal 
to the cylinder which is the axis of the rotation. Therefore the sphere is tangent to the cylinder.

\subsubsection{Regular cyclides in $ß^3$ }\label{sus:2}
Each regular Dupin cyclide $M$ is, for a suitable metric of constant
 curvature $1$, the tubular neighborhood of a geodesic ${\mathcal G}_1$
 of $ß^3$ (see \cite{La-Wa1}). Seeing $ß^3$ as the unit sphere of
 an euclidian space $\EE^4$ of dimension $4$, it is the intersection 
of a $2$-plane $P_1$ with $ß^3$. Let $P_2$ be the plane orthogonal 
to $P_1$ in $\EE^4$. Then it is also a tubular neighborhood of the
 geodesic ${\mathcal G}_2 = P_2 \cap ß^3$.
We can define (for this metric) the symmetries with respect to the two
  spheres containing  respectively $m$ and
${\mathcal G}_1$ and $m$ and ${\mathcal G}_2$.
Let ${\cal R} $ be the composition of these symmetries.
Then $\cal R$ preserves the cyclide but also the ``helices" on the cyclide.
It should therefore also preserve the osculating sphere to the helix, 
which, as it is not normal to the cyclide, has to be tangent to it.
The ``helix" is therefore a Darboux curve. There are enough ``helices"
 to be sure we got all the Darboux curves.

The last family of Dupin cyclides is formed of conformal images of 
cones of revolution (see \cite{La-Wa1}).
On can deal with cones  of revolution as we did with regular cyclides,
 using a sphere orthogonal to the axis and $m$ (it belongs to  the 
pencil whose limit points are the singular points) and a sphere containing the axis and $m$.

We will, using the general dynamical properties of Darboux curves, prove (see Proposition \ref{caract_Dupin_via_Darboux}) that a surface is a Dupin cyclide 
if and only if its Darboux curves are the leaves of the foliation $\tilde{\fol}$ of $V(M)$.

\section{Differential Equation of Darboux Curves in a Principal Chart}\label{sc:4}

Consider a local principal  chart $(u,v)$ in a surface $\mathbb
M\subset\mathbb R^3$. The first and second fundamental forms are
denoted by
$$I=Edu^2+ Gdv^2,\;\;\; \;\;\; II=I=edu^2+ gdv^2 $$
and the principal curvatures are $k_1=e/E$ and $k_2=g/G$.

\begin{prop} \label{prop:dpc} Let $(u,v)$ be a principal chart. Let $c$ be
a curve parametrized by arc length $s$ making an angle $\alpha(s)$
with the principal direction $(1,0)$. For a Darboux line $c$ the
following differential equation is verified

\begin{equation}\label{eq:dpc} \aligned  3(k_1-k_2)\sin\alpha\cos\alpha\;\frac{d\alpha}{ds} =& \frac{1}{\sqrt{E}} \frac{\partial k_1}{\partial u}
 \cos^3\alpha+ \frac{1}{\sqrt{G}}\frac{\partial k_2}{\partial v}
\sin^3\alpha\\
=&  \frac{\partial k_1}{\partial  s_1}
 \cos^3\alpha+  \frac{\partial k_2}{\partial s_2}
\sin^3\alpha.\endaligned \end{equation}
Here $(u^\prime, v^\prime)=(\frac{\cos\alpha}{\sqrt{E}}, \frac{\sin\alpha}{\sqrt{G}} )$.
\end{prop}
\begin{proof} Consider a principal chart $(u,v)$ such that $v=cte$ are
 the leaves of the principal foliation ${\mathcal P}_1$.

Let
$c(s)=(u(s),v(s)) $ be a regular curve parametrized by arc length $s$. 
So we can write  $c^\prime(s)=(u^\prime, v^\prime)=
(\frac{\cos\alpha}{\sqrt{E}}, \frac{\sin\alpha}{\sqrt{G}})$, defining
 a direction $\alpha$ with respect to principal foliation ${\mathcal P}_1$, horizontal foliation.

We have the following classical relations:
$$\aligned k_n(\alpha)=&k_1\cos^2\alpha+k_2\sin^2\alpha\\
k_g=&\frac{d\alpha}{ds}+k_g^1\cos\alpha+k_g^2\sin\alpha\\
\tau_g=& (k_2-k_1) \cos\alpha\sin\alpha. \endaligned $$

Here
$k_g^1=(k_g)_{v=const} $  and $k_g^2=(k_g)_{u=const} $ are the
geodesic curvatures of the coordinates curves.

Therefore,

$$ \aligned \frac{dk_n}{ds}=& \frac{1}{\sqrt{E}} \frac{\partial k_1}{\partial
u} \cos^3\alpha+\frac{1}{\sqrt{G}} \frac{\partial k_1}{\partial v}
\cos^2\alpha\sin\alpha+\frac{1}{\sqrt{E}}  \frac{\partial
k_2}{\partial
u} \cos \alpha\sin^2\alpha\\
+&\frac{1}{\sqrt{G}}  \frac{\partial k_2}{\partial v}
\sin^3\alpha+ 2(k_2-k_1)\cos\alpha\sin\alpha\frac{d\alpha}{ds}\endaligned $$

The differential equation of Darboux curves  is given by $k_n^\prime
+k_g\tau_g=0$ and so it follows that:

$$\aligned
 & [k_g^1(k_2-k_1) +\frac{1}{\sqrt{G}} \frac{\partial
k_1}{\partial v} ]\cos^2\alpha\sin\alpha+
[k_g^2(k_2-k_1)+\frac{1}{\sqrt{E}} \frac{\partial k_2}{\partial u}
]\cos\alpha\sin^2\alpha\\
&+ 3 (k_2-k_1)\cos\alpha\sin\alpha\frac{d\alpha}{ds}+
\frac{1}{\sqrt{E}} \frac{\partial k_1}{\partial u} \cos^3\alpha+
\frac{1}{\sqrt{G}} \frac{\partial k_2}{\partial v} \sin^3\alpha=0
\endaligned $$

In  any orthogonal chart ($F=0$) we have that $G_u= 2 G\sqrt{E}
k_g^2=(k_g)_{u=cte)}$ and $E_v= -2 E\sqrt{G} k_g^1=(k_g)_{v=cte}$.    Also the Codazzi
equations in a principal chart are given by:

$$\frac{\partial k_1}{\partial v}=\frac{E_v}{2E}(k_2-k_1), \;\; \frac{\partial k_2}{\partial
u}=\frac{G_u}{2G}(k_1-k_2).$$

See Struik's book  \cite[pages 113 and 120]{struik}.

Therefore,

$$\aligned    k_g^1(k_2-k_1) +\frac{1}{\sqrt{G}} \frac{\partial k_1}{\partial
v}  &= -\frac{E_v}{2E\sqrt{G}}(k_2-k_1)+\frac{1}{\sqrt{G}}\frac{E_v}{2E}(k_2-k_1)=0\\
  k_g^2(k_2-k_1)+\frac{1}{\sqrt{E}} \frac{\partial k_2}{\partial
u}  &= \;\;  \frac{G_u}{2G\sqrt{E}}(k_2-k_1)+
\frac{1}{\sqrt{E}}\frac{G_u}{2G}(k_1-k_2)=0. \endaligned$$ This
ends the proof.
\end{proof}

\begin{rema}
A curve $c(s)$ has contact of third order with the associated
  osculating sphere, tangent to the surface, when
$$  \langle c^\prime,c^\prime\rangle 
[  2\langle N^\prime,c^{\prime\prime}  \rangle+ \langle N^{\prime\prime},c^{\prime }\rangle]-
3  \langle c^{\prime }, N^\prime\rangle    \langle c^\prime ,c^{\prime
\prime} \rangle=0.$$

This equation can be used to obtain the differential equation of Darboux
 curves in any chart $(u,v)$. See \cite{Sa}.

\end{rema}

%

\section{A plane-field on $V(M)$}\label{sc:6}
The two tangents to the two Darboux orbits through the point
 $(m, \alpha)\in V(M)$ define a plane in $T_{m,\alpha}V(M)$.
 The ensemble of  these planes define a plane-field $\mathcal P$.

\begin{prop}\label{prop:pfdi}
The plane-field $\mathcal P$ is integrable if and only if
 $$(\xi_1)\theta_2=-\frac 16 \theta_1\theta_2,\hskip 2cm (\xi_2)\theta_1= \frac 16 \theta_1\theta_2.$$
Here  $\xi_i$ are the conformal vector fields and $\theta_i$ are the principal conformal curvatures.

\end{prop}


\begin{proof}
Consider in the unitary tangent bundle the suspension of Darboux curves.

These curves are defined by the following vector field

$${\mathcal D}_1=\frac{\cos\alpha}{\sqrt{E}}\frac{\partial}{\partial u}+\frac{\sin\alpha}{\sqrt{G}}
\frac{\partial}{\partial v}+
[\frac { \partial {k_1} /{\partial u}}{3\sqrt{E}(k_1-k_2)}\frac{\cos^2\alpha}{\sin\alpha}+
\frac { \partial {k_2} /{\partial v}}{3\sqrt{E}(k_1-k_2)}
\frac{\sin^2\alpha}{\cos\alpha}]\frac{\partial}{\partial \alpha} .$$

Consider the involution $\varphi(u,v,\alpha)=(u,v, -\alpha)$
 and the induced vector field $D_2=\varphi_* (D_1)$.

So it follows that:
$${\mathcal D}_2=\frac{\cos\alpha}{\sqrt{E}}\frac{\partial}{\partial u}-
\frac{\sin\alpha}{\sqrt{G}}\frac{\partial}{\partial v}
-[-\frac { \partial {k_1} /{\partial u}}{3\sqrt{E}(k_1-k_2)}\frac{\cos^2\alpha}{\sin\alpha}+
\frac { \partial {k_2} /{\partial v}}{3\sqrt{E}(k_1-k_2)}
\frac{\sin^2\alpha}{\cos\alpha}]\frac{\partial}{\partial \alpha}.$$

Consider the plane field, {\it Darboux plane field}, defined by $\{{\mathcal D}_1,{\mathcal D}_2\}$.

In the sequence it will be obtained the conditions of integrability of this plane field.

In order to simplify the calculations the following changes will be developed.

First, consider the new pair of vector fields
$\tilde {\mathcal D}_1={\mathcal D}_1+{\mathcal D}_2$ and
 $\tilde {\mathcal D}_2={\mathcal D}_1-{\mathcal D}_2$ and obtain:

$$\aligned \tilde {\mathcal D}_1=& \frac{2\cos\alpha}{\sqrt{E}}\frac{\partial}{\partial u}+
[\frac {2 \partial {k_1} /{\partial u}}{3\sqrt{E}(k_1-k_2)}\frac{\cos^2\alpha}{\sin\alpha}
 ]\frac{\partial}{\partial \alpha} \\
\tilde {\mathcal D}_2=&  \frac{2\sin\alpha}{\sqrt{G}}\frac{\partial}{\partial v}+
[
\frac {2 \partial {k_2} /{\partial v}}{3\sqrt{E}(k_1-k_2)}\frac{\sin^2\alpha}{\cos\alpha}]\frac{\partial}{\partial \alpha}.\endaligned $$

Next consider:
$$\aligned \bar {\mathcal D}_1=& \frac{2 }{\sqrt{E}}\frac{\partial}{\partial u}+
[\frac {2 \partial {k_1} /{\partial u}}{3\sqrt{E}(k_1-k_2)}\frac{\cos \alpha}{\sin\alpha}
 ]\frac{\partial}{\partial \alpha} \\
\bar {\mathcal D}_2=&  \frac{2}{\sqrt{G}}\frac{\partial}{\partial v}+
[
\frac {2 \partial {k_2}/{\partial v}}{3\sqrt{E}(k_1-k_2)}
\frac{\sin \alpha}{\cos\alpha}]\frac{\partial}{\partial \alpha} .\endaligned $$

Consider the unitary vector fields $X_i$, 
 the conformal vector fields $\xi_i$ and the principal conformal curvatures $\theta_i$.

$$\aligned X_1=&\frac{1}{\sqrt{E}}\frac{\partial}{\partial u},
\;\;\;\hskip 2cm X_2= \frac{1}{\sqrt{G}}\frac{\partial}{\partial v}\\
\xi_1=&\frac{2X_1}{k_1-k_2},\;\;\; \hskip 2cm\xi_2=\frac{2X_2}{k_1-k_2}\\
\theta_1=& \frac{4 (X_1)k_1}{{(k_1-k_2)}^2}\;\;\;\hskip 2cm
 \theta_2=  \frac{4 (X_2)k_2}{{(k_1-k_2)}^2}, \;\; X_i(k_i)=D k_i(X_i).
\endaligned$$

Observing that $\frac{ \partial {k_1} /\partial u}{\sqrt{E}}=(X_1)k_1$
 and  $\frac{\partial{ k_2} /\partial v}{\sqrt{G}}=(X_2)k_2 $ we obtain a new base defined by:

$$\aligned   {\mathcal D}_1^c=&  \xi_1+
 \frac 16 \theta_1\frac{\cos \alpha}{\sin\alpha}
 \; \frac{\partial}{\partial \alpha} \\
 {\mathcal D}_2^c=& \xi_2+\frac 16\theta_2
 \frac{\sin \alpha}{\cos\alpha} \;\frac{\partial}{\partial \alpha} .\endaligned $$

 Recall that:

 $$\aligned &[fX,gY]=  fg[X,Y]+(X.g)f Y- (Y.f) g X\\
 &[\xi_1,\xi_2]=   -\frac 12 \theta_2\,\xi_1-\frac 12  \theta_1 \,\xi_2.\endaligned $$

 So it follows that:
{\small
 $$\aligned &[{\mathcal D}_1^c, {\mathcal D}_2^c]=  [\xi_1,\xi_2]+[\xi_1,\frac 16\theta_2
 \frac{\sin \alpha}{\cos\alpha}
 \;\frac{\partial}{\partial \alpha}]+[\frac 16 \theta_1\frac{\cos \alpha}{\sin\alpha}
 \; \frac{\partial}{\partial \alpha}, \xi_2]\\
 +& [\frac 16 \theta_1\frac{\cos \alpha}{\sin\alpha}
 \; \frac{\partial}{\partial \alpha}, \frac 16\theta_2
 \frac{\sin \alpha}{\cos\alpha} \;\frac{\partial}{\partial \alpha}]\\
 =& -\frac 12 \theta_2\,\xi_1-\frac 12 
 \theta_1 \,\xi_2 +\frac 16 (\xi_1)\theta_2\frac{\sin\alpha}{\cos\alpha}\;
\frac{\partial}{\partial \alpha}-
 \frac 16 (\xi_2)\theta_1\frac{\cos\alpha}{\sin\alpha}\;
\frac{\partial}{\partial \alpha}-
\frac{1}{18}\frac{\theta_1\theta_2}{\sin\alpha\cos\alpha} 
\;\frac{\partial}{\partial \alpha}\\
 =& -\frac 12 \theta_2\,\xi_1-\frac 12  \theta_1 \,\xi_2 +
 \frac 16 [ (\xi_1)\theta_2\frac{\sin\alpha}{\cos\alpha} -
 (\xi_2)\theta_1\frac{\cos\alpha}{\sin\alpha} -
\frac{1}{3}\frac{\theta_1\theta_2}{\sin\alpha\cos\alpha} ]\;\frac{\partial}{\partial \alpha}
 \endaligned$$
}
 Now consider the determinant

 $$\aligned \left| \begin{matrix} 1& 0 & \frac 16\theta_1\frac{\cos\alpha}{\sin\alpha}\\
 0& 1& \frac 16\theta_2\frac{\sin\alpha}{\cos\alpha}\\
 -\frac 12 \theta_2& -\frac 12 \theta_1 & \lambda\end{matrix}\right|,\;\;
\lambda=     \frac 16 [ (\xi_1)\theta_2\frac{\sin\alpha}{\cos\alpha} -
 (\xi_2)\theta_1\frac{\cos\alpha}{\sin\alpha}
 -\frac{1}{3}\frac{\theta_1\theta_2}{\sin\alpha\cos\alpha} ]
 \endaligned$$

 Evaluating this determinant it is obtained:

 $$\frac{-3[(\xi_1)\theta_2+(\xi_2)\theta_1]\cos 2\alpha+[
  3 (\xi_1)\theta_2-3(\xi_2)\theta_1  +\theta_1\theta_2 ]}{36\sin\alpha\cos\alpha}$$

 So the integrability conditions are given by:

 $$(\xi_1)\theta_2+(\xi_2)\theta_1=0,\hskip 1cm 3 (\xi_1)\theta_2-3(\xi_2)\theta_1  +\theta_1\theta_2 =0.$$

This is equivalent to:\; $(\xi_1)\theta_2=-\frac 16 \theta_1\theta_2\;$ and $\;(\xi_2)\theta_1= \frac 16 \theta_1\theta_2.$
\end{proof}

\section{Darboux lines near a Ridge Point}\label{sc:7}

Consider a surface $ M$ and   a principal chart $(u,v)$ such 
that the horizontal foliation ${\mathcal P}_1$ is that associated to the principal curvature $k_1$.
\begin{defi}
A non umbilic point $p_0=(u_0,v_0)$ is called a {\it ridge} point
 of the principal foliation ${\mathcal P}_1$ if $\frac{\partial k_1}{\partial u}(p_0)=0$.
\end{defi}
The  ridges for the principal foliation ${\mathcal P}_2$ are 
characterized by $\frac{\partial k_2}{\partial v}(p_0)=0$.
\begin{defi}
A ridge point relative to  ${\mathcal P}_1$ is called  {\it zigzag},
 respectively {\it beak to beak},  when
 $\sigma_1(p_0)=\frac{\partial^2 k_1}{\partial u^2}(p_0)/( k_1(p_0) - k_2(p_0))  <0,$
 respectively,  $\sigma_1(p_0)>0$.
%

A ridge point $p_0$ relative to  ${\mathcal P}_2$ is called 
 {\it zigzag}, respectively {\it beak to beak},  when 
$\sigma_2(p_0)=\frac{\partial^2 k_2}{\partial v^2}(p_0)/ (k_2(p_0)-k_1(p_0))<0,$
 respectively $\sigma_2(p_0)>0$.
\end{defi}

The ridge points are associated to inflections of the principal curvature
 lines, to  the singularities of the focal set of the surface and 
also with the singularities of  the boundary of $V(M)$, the space
 of spheres tangent to $M$. See \cite{Po} for an introduction to
  ridges and also \cite{Gu} in his  study of geometric optic and
 applications in construction of eye lens.

A practical way to see the type of a ridge point is given by the following proposition.

\begin{prop}\label{prop:rzigzag} Consider a surface of class
 $C^r, \; r\geq 4$ parametrized by the graph $(u,v,h(u,v))$
where
$$\aligned h(u,v)=&\frac{k_1}2 u^2+\frac {k_2} 2 v^2+\frac a6 u^3+\frac d2
u^2v+\frac b2 uv^2+\frac c6 v^3\\
+&\frac{ A}{24}u^4+\frac{ B}6 u^3v+\frac C4 u^2v^2+\frac D6
uv^3+\frac{E}{24}  v^4+h.o.t\endaligned$$
Then $(0,0)$ is a ridge point for ${\mathcal P}_1$ when $a=0$. Also
$\sigma_1= [\frac{A-3k_1^3}{ k_1-k_2 }+\frac{2d^2}{(k_1-k_2)^2}].$

Corresponding to the foliation ${\mathcal P}_2$ the point $(0,0)$ is a ridge point when $c=0$. Also
$\sigma_2(0)=[\frac{E-3k_2^3}{ k_2-k_1}+\frac{2b^2}{(k_2-k_1)^2}]$.
\end{prop}

\begin{proof} Straightforward calculations shows that
the principal curvatures in a neighborhood of $(0,0)$ are given
by:

\begin{equation}\label{eq:rz}
\aligned k_1(u,v)=&k_1+au+dv+\frac
12(A-3k_1^3+\frac{2d^2}{k_1-k_2})u^2+(B-2\frac{bd}{k_2-k_1})uv
\\+& \frac
12(C-k_1k_2^2
- \frac{2b^2}{k_2-k_1} )v^2+h.o.t.\\
 k_2(u,v)=&k_2+bu+cv+\frac
12(C-k_1^2k_2+\frac{2d^2}{k_2-k_1})u^2+(D+2\frac{bd}{k_2-k_1})uv\\
+&\frac 12(E-3k_2^3+\frac{2b^2}{k_2-k_1} )v^2+h.o.t.\endaligned
\end{equation}
So the result follows.
\end{proof}

\begin{prop} \label{prop:ridgeboundary} Let $p_0$ be a ridge
 point of $M$ corresponding to principal foliation ${\mathcal P}_1$ such that
$\sigma_1(p_0)\ne 0$. Then the ridge set $R$ containing $p_0$
 is locally a regular curve transversal to  ${\mathcal P}_1$
and the boundary of $V(M)$  corresponding to $\alpha=0$ has a
 cuspidal edge along $\pi^{-1}(R)$. See Fig. \ref{fig:ridgeVM}.
 Analogously for the ridges associated to the principal foliation ${\mathcal P}_2$.
 \begin{figure}[ht]
 \psfrag{p1}{$p_1$}
\begin{center}
\includegraphics[scale=0.60]{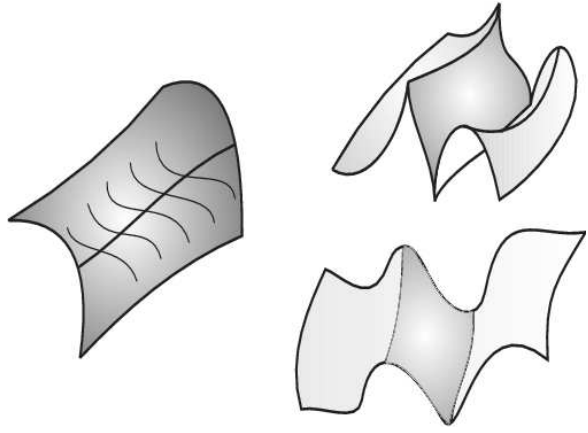}
\caption{Ridges and Singularities of the boundary of $V(M)$ \label{fig:ridgeVM} }
\end{center}
\end{figure}
\end{prop}

\begin{proof} In the parametrization given in Proposition \ref{prop:rzigzag} it
 follows that the at $p_0=(0,0)$ the principal direction corresponding to
 ${\mathcal P}_1$ is $e_1=(1,0)$ and the ridge set is parametrized,
 according to equation \eqref{eq:rz}, by:
$(u(v),v)=(-\frac{d(k_1-k_2)}{ (A-3k_1^3)(k_1-k_2)+ 2d^2 }v+O(v^2),v).$

Next consider a principal chart $(u,v)$.
The set of spheres $V(M)\subset \Lambda^4$ is parametrized by
 $\sigma(u,v,\alpha)=k_n(\alpha) m(u,v)+ N(u,v)$ with
 $k_n(\alpha) =k_1(u,v)\cos^2\alpha+k_2(u,v)\sin^2\alpha $,
 ${\mathcal L}(m,m)=0$ and $N_u=-k_1 m_u, \; N_v=-k_2m_v.$ See equation \eqref{sig_sig}.
We have that
$$\aligned \sigma_u=&[\frac{\partial k_1}{\partial u}\cos^2\alpha+
\frac{\partial k_2}{\partial u}\sin^2\alpha]m+(k_n-k_1)m_u\\
\sigma_v=& [\frac{\partial k_1}{\partial v}\cos^2\alpha+
\frac{\partial k_2}{\partial v}\sin^2\alpha]m+(k_n-k_2) m_v\\
\sigma_\alpha=&[ (k_2-k_1)\cos \alpha \sin \alpha] m\endaligned$$
So   $D\sigma$ has rank 3 for $\alpha\in (0,\frac{\pi}2)$.

The boundary of $V(M)$ is parametrized by $\alpha=0$ and $\alpha=\pi/2$ and so
$\sigma_1(u,v )=k_1(u,v) m(u,v)+ N(u,v)$  and 
 $\sigma_2(u,v)=k_2(u,v) m(u,v)+ N(u,v).$
The map $\sigma_1$ has rank 1 at the ridges and 
so we have the structure of cuspidal edges on the boundary of $V(M)$.
\end{proof}

\begin{prop}\label{prop:dlocal} Let $p_0$ be a non umbilic and disjoint from the ridge set.
 The Darboux curves  tangent to the principal lines are given by cuspidal curves
 $r_1(t)=(\frac 32 \frac{\partial k_1}{\partial u}(k_1-k_2) t^2+\cdots, 
  \frac{\partial k_1}{\partial u}(k_1-k_2)  t^3+\cdots)$ and 
 $r_2(t)=( \frac{\partial k_2}{\partial v}(k_2-k_1)t^3+\cdots, 
\frac 32 \frac{\partial k_2}{\partial v}(k_2-k_1) t^2+\cdots)$.
 The behavior of all the Darboux curves passing through $p_0$ 
is as shown in the Fig. \ref{fig:dregular}.
 \begin{figure}[ht]
 \psfrag{p1}{$p_1$}
    \psfrag{p2}{$p_2$}
  \psfrag{p3}{$p_3$}
 \psfrag{p4}{$p_4$}
\begin{center}

\includegraphics[scale=0.50]{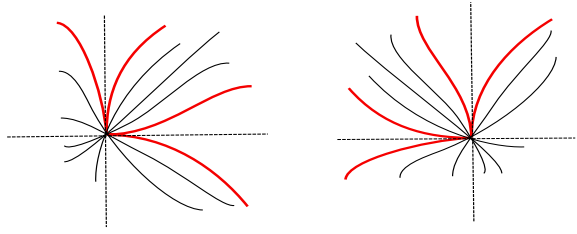}
\caption{Darboux curves through a non ridge point \label{fig:dregular} }
\end{center}
\end{figure}
\end{prop}

\begin{proof} Consider the vector field $X$ defined by the differential equation:

$$\aligned u^\prime =& \frac{1}{\sqrt{E}}\cos\alpha[3(k_1-k_2)\sin\alpha\cos\alpha]\\
 v^\prime =&  \frac{1}{\sqrt{G}}\sin\alpha[3(k_1-k_2)\sin\alpha\cos\alpha]\\
\alpha^\prime=&  \frac{1}{\sqrt{E}} \frac{\partial k_1}{\partial u}
 \cos^3\alpha+ \frac{1}{\sqrt{G}}\frac{\partial k_2}{\partial v}
\sin^3\alpha.  \endaligned
$$

The projections of the integral curves of $X$ in the coordinates
 $(u,v)$ are precisely the Darboux curves.

For any initial condition $(0,\alpha_0)$, with $\alpha_0\ne k\pi/2$,
 the integral curves of $X$ are transversal to the axis $\alpha$ and so has a regular projection.
For $\alpha_0=n\pi/2$ and $\sigma_2\ne 0$, the integral curves of
$X$ are tangent to the axis $\alpha$ and the  projections are of
 cuspidal type.
For $\alpha=n\pi $ direct calculations gives:
 $(u(t),v(t))=(\frac 32 \frac{\partial k_1}{\partial u}(k_1-k_2)t^2+\cdots,
 (-1)^n \frac{\partial k_1}{\partial u}(k_1-k_2)t^3+\cdots)$.
Now observe that the projection of the integral curves passing
 through $(0,0,0)$ and $(0,0,\pi)$ are both tangent to semi axis of $u$.
\end{proof}

\begin{theo}\label{th:zigbec} Let $R$ be an arc of ridge points 
  transversal
to the corresponding principal foliation, i.e., suppose
 that $\sigma_i(p)\ne 0 $ for every $p\in  R$. Then  there are two types of
behavior for the Darboux curves  near the ridge set, the first
is the zigzag and the
other is the beak to beak.

\end{theo}

\begin{figure}[htb]
\begin{center}
 \includegraphics[scale=0.5]{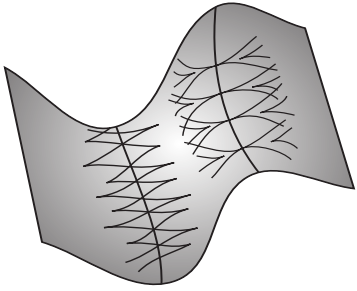}
\end{center}
\caption{ Darboux curves  near regular curve of ridges: zigzag and beak to beak.  \label{fig:dridge} }
\end{figure}

\begin{proof} Consider the  vector field $X$ as in the proof of proposition \ref{prop:dlocal}.


We will consider only the ridge set corresponding to ${\mathcal P}_1$. 
For the other principal foliation the analysis  is similar.
The ridge set is defined by the equations $\frac{\partial k_1}{\partial u}(u,v)=0$
 and $\frac{\partial k_2}{\partial v}(u,v)=0$, one for the corresponding principal foliation.

The singularities of $X$ are defined by $(U(v),v),0)$ and $(u,V(u),\frac{\pi}2)$, 
 where  $\frac{\partial k_1}{\partial u}(U(v),v)=0$ and $\frac{\partial k_2}{\partial v}(u,V(u))=0$.

To simplify the notation suppose a singular point $(0,0,0)$ of the ridge
 transversal to the principal foliation ${\mathcal P}_1$ and $E(0)=G(0)=1$.

It follows that:

$$\aligned DX(0)=&\left(\begin{matrix} 0 & 0 & 3(k_1-k_2) \\
0 & 0 & 0 \\ \frac{{\partial}^2 k_1}{\partial u^2} & 0 & 0 \end{matrix} \right) \endaligned$$

The eigenvalues of $DX(0)$ are:
$$\lambda_1=0,  \lambda_2=\frac{1}{\sqrt{3}} \sqrt{ \frac{{\partial}^2 k_1}{\partial u^2} / (k_1-k_2)},
\lambda_2= -\frac{1}{\sqrt{3}}\sqrt{ \frac{{\partial}^2 k_1}{\partial u^2}/ (k_1-k_2)}.$$

By invariant manifold theory, when $\lambda_2\lambda_3=-
 \frac 13\frac{{\partial}^2 k_1}{\partial u^2} /(k_1-k_2) =-\frac 13 \sigma_1(0) <0$,
 the singular set of $X$ (ridge set) is normally hyperbolic and there
 are stable and unstable surfaces, normally hyperbolic along the
 singular set. This implies that there is a lamination 
(continuous fibration) along the ridge set and the fibers
 are the  Darboux curves. Also  the prolonged Darboux curves are o class $C^1$ along the ridge set.

So the Darboux curves are as shown in Fig. \ref{fig:dridge}, 
center and right. That is, there are Darboux curve crossing the
 ridge, tangent to the principal lines, and the prolonged Darboux 
curves are $C^1$ along the ridge set.

In the case when $  \sigma_1(0)= \frac{{\partial}^2 k_1}{\partial u^2} /(k_1-k_2) <0$ 
 the non zeros eigenvalues of $DX(0)$ are purely complex   and so the singular set is not normally hyperbolic.

In this case we are in the hypothesis of Roussarie Theorem,
 \cite[Theorem 20, page 59]{roussarie}, so there is a local first
 integral in a neighborhood of the ridge set. The level sets of 
this first integral are cylinders and the integral curves (helices)
 in each cylinder when projected in the surface $M$ has a cuspidal
 point exactly when  helix cross the section $\alpha=0$. This produces the zigzag.

There are no Darboux curves tangent to the principal direction
 $e_1$ along the ridge set in this case.
\end{proof}

\section{Darboux curves on general cylinders,   cones and surfaces of revolution  } \label{sc:8}

Darboux curves on general cones where already studied by Santaló (\cite{sa2}).
In a similar way, one can study Darboux curves on cylinders and surfaces of revolution.
This is not a coincidence. The three type of surfaces are canal surfaces
 corresponding to a curve $\gamma  \subset \Lambda^4$ which is also contained
 in a $3$-dimensional subset of  $\ll^5$. Depending on the subspace, this intersection
 is either a copy of $\Lambda^2$, a unit sphere $ß^3$ or a $2$-dimensional 
cylinder (see \cite{Da}, \cite{mn},  \cite{Ba-La-Wa}). The latter condition defines
 conformal images of general cones, general cylinders and surfaces of revolution.

These surfaces can be obtained imposing conformally  invariant local conditions.

Recall that, assuming that $S$ is a surface which is {\it
umbilic free}, that is, that the principal curvatures $k_1(x)$ and $k_2(x)$
of $S$ are different at any point $x$ of $S$. Let $X_1$ and $X_2$ be unit
vector fields tangent to the curvature lines corresponding to,
respectively, $k_1$ and $k_2$. Throughout the paper, we assume that $k_1
> k_2$. Put $µ= (k_1 - k_2)/2$. Since more than 100 years, it is known
(\cite{tresse}, see also \cite{csw}) that the vector fields $\xi _i = X_i/µ$
and the coefficients $\theta_i$ ($i = 1,2$) in
\begin{equation*}
 [\xi _1, \xi_ 2] = -\frac{1}{2}\left( \theta_2\xi _1 + \theta_1 \xi_2\right)
\end{equation*}
are invariant under arbitrary (orientation preserving) conformal
transformation of $\mathbb R^3$. (In fact, they are invariant under
arbitrary conformal change of the Riemannian metric on the ambient space).
Elementary calculation involving Codazzi equations shows that
\begin{equation*}
\theta_1 = \frac{1}{µ^2}\cdot X_1(k_1)\quad\text{and}\quad\theta_2 =
\frac{1}{µ^2}\cdot X_2(k_2).
\end{equation*}
The quantities $\theta_i$ ($i=1,2$) are called {\it conformal principal
curvatures} of $S$.

Canal surface are characterized locally,  \cite{Da}, \cite{H-J}, \cite{mn},
 \cite{Ba-La-Wa}, by the following propositions.

\begin{prop}\label{prop01canal} A surface $S$ is (a piece of) of a canal
 if and only if one of its
conformal principal curvatures, say $\theta_2$, is equal to zero.

\end{prop}

\begin{prop}\label{prop02canal}
Imposing moreover that the other conformal curvature, say $\theta_1$,
 is constant along characteristic circles characterizes the surface
 as  one of the three families
 above (cone, cylinder and surfaces of revolution).
\end{prop}

\begin{prop}\label{prop:clr}  Let  $M$ be a surface and $(u,v)$ be a principal
 chart  such that $\theta_1(u,v)=\theta_1(u)$ and $\theta_2(u,v)=0$. 
Let $A(u)=exp[\int \frac{k_1^\prime}{k_1-k_2}du] $ and $\alpha\in (0,\pi)$ be an angle.
Then the function ${\mathcal J}(u,\alpha)=A(u)\cos^3\alpha$ is a first integral
 of the Darboux curves.
Moreover in the region $A_c=\pi(M_c)=\{(u,v): u\in M_c\} $,  $M_c={\mathcal J}^{-1}(c)$,
 the Darboux curves are defined by the implicit differential equation
$$ c^{2/3}G dv^2- E (A^{2/3} -c^{2/3})du^2=0.$$

\end{prop}
\begin{proof} The differential equation \eqref{eq:dpc} is simplified to
$$u^\prime=\frac{\cos \alpha}{\sqrt{E}},\;\;v^\prime=
\frac{\sin\alpha}{\sqrt{G}},\;\;\alpha^\prime=  \frac{1}{3\sqrt{E}}
\frac{k_1^\prime}{k_1-k_2}\frac{\cos^2\alpha}{\sin\alpha}.$$

So it follows that $\frac{d\alpha}{du}= \frac 13
 \frac{k_1^\prime}{k_1-k_2}\frac{\cos\alpha}{\sin\alpha}$ which 
is a equation of separable variables. Direct integration leads to the
 first integral $\mathcal J$ as stated.

To obtain the implicit differential equation solve the 
equation ${\mathcal J}(u,v)=c$ in function of $\cos\alpha$ and observe that
 $\frac{dv}{du}=\frac{\sqrt{E}}{\sqrt{G}}\frac{\sin\alpha}{\cos\alpha}.$
\end{proof}

\begin{prop}\label{prop:clrint}  Let  $M$ be a surface and
 $(u,v)$ be a principal chart  such that $\theta_1(u,v)=\theta_1(u)$ and $\theta_2(u,v)=0$. 
Then the plane-field $\mathcal P$ is integrable.
\end{prop}

\begin{proof} Direct from the characterization of integrability 
of $\mathcal P$ established  in Proposition \ref{prop:pfdi}.
\end{proof}





 \subsection{Darboux lines on surfaces of Revolution as Canal Surface}

Consider an one parameter family of spheres of radius $r(u)$ with
center at $(0,0,u)$.

The envelope of this family is a canal surface and can be
parametrized by:

$$ H(u,v)=  r(u)\cos\beta(u)( \cos v,  \sin v ,0)+(0,0, u-r(u)\sin\beta(u)), $$
where $\cos\beta(u)= \sqrt{1-r^\prime(u)^2}, \;\; \sin\beta=r^\prime,
\;\;|r^\prime (u)|<1, \;\;\beta\in (-\pi/2,\pi/2)$.

The normal unitary to the surface is:

$$N=( -  \cos\beta(u)\cos v,  -\cos\beta(u)\sin v,   \sin\beta(u)).$$

The coefficients of the first and second fundamental forms of
$H$ are given by:
$$\;\;E(u,v)= \frac{(1-{r^\prime}^2-r
r^{\prime\prime})^2}{1-{r^\prime}^2},
 \;\; \;\;\;\;\; F(u,v)=0, \;\;\;\; G(u,v)=r^2 (1-{r^\prime}^2)$$
$$e(u,v)= -\frac{r^{\prime\prime} (1-{r^\prime}^2-r r^{\prime\prime})}{1-{r^\prime}^2}
 ,\;\; \;\;\ f(u,v)=0,\;\;\; \;\;\;g(u,v)=  r (1-{r^\prime}^2).
$$

The principal curvatures are given by:

$$k_1(u,v)= - \frac{ r^{\prime\prime} }{ 1-{r^\prime}^2-r r^{\prime\prime}}\;\;\;\;\;
k_2(u,v)= \frac 1{r}. $$

It will be assumed that $k_2>  k_1$ and the surface is free of umbilic points.

The ridge set is defined by the equation
$$R(u,v)=\frac{\partial k_1}{\partial u}=   r^{\prime\prime\prime}(1-{r^\prime}^2
)+3r^\prime {r^{\prime\prime}}^2=0.$$

\begin{prop}\label{prop:revolution} The Darboux lines on
surfaces of revolution, free of umbilic points,  can be integrated by quadratures. The
function
$$
{\mathcal I}(u,\alpha)=  { r\cos\beta}{(k_1-k_2)}\cos^3\alpha=h(u)(k_1-k_2)\cos^3\alpha,\;\;$$
 is a
first integral  of the differential equation of
 Darboux lines. Here $h$ is the distance of the point 
of the surface to  the axis of revolution.

 Moreover, if $R^\prime(u)<0 $ the ridge is zigzag. If
$R^\prime(u)>0 $ the ridge is beak to beak.   See Fig.
\ref{fig:drev}.


\end{prop}

\begin{figure}[htb]
\begin{center}

 \includegraphics[scale=0.5]{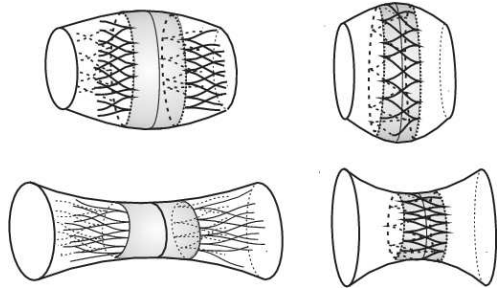}
\end{center}
\caption{ Darboux lines near regular ridges on surfaces of
revolution \label{fig:drev} }
\end{figure}

\begin{proof}  In the principal chart $(u,v)$ the differential
equation of Darboux lines is given by

   $$ u^\prime =  \frac{1}{\sqrt{E}}\cos\alpha,\;\; v^\prime =   \frac{1}{\sqrt{G}}\sin\alpha,\;
 \alpha^\prime = \frac 1{3\sqrt{E}} [\frac{ k_1^\prime}{k_1-k_2}]\frac{\sin^2\alpha}{\cos\alpha} $$

So it follows that $3\frac{\sin\alpha}{\cos\alpha}d\alpha=\frac{ k_1^\prime}{k_1-k_2}du$.

Now observe that
$$\aligned \int \frac{ k_1^\prime}{k_1-k_2}du=&\int
 \frac{ k_1^\prime-k_2^\prime}{k_1-k_2}du+\int \frac{ k_2^\prime}{k_1-k_2}du
= \ln(k_2-k_1)+\int \frac{ k_2^\prime}{k_1-k_2}du\\
\int \frac{ k_2^\prime}{k_1-k_2}du=& 
 \int   [\frac{r^\prime}{r}    + \frac{r^\prime r^{\prime\prime}}{1-{r^\prime}^2}]du=
 \ln r{(1-{r^\prime}^2)}^{1/2}\endaligned
$$
Therefore it follows that

$$\aligned   {\mathcal I}(u,v,\alpha)= &
 {r(1-{r^\prime}^2)^{1/2}}{(k_1-k_2)}\cos^3\alpha \\
=&   {r\cos\beta}{(k_1-k_2)}
\cos^3\alpha=h(u)(k_1-k_2)\cos^3\alpha \endaligned$$
is the first integral. The analysis of behavior near ridges
 follows from Theorem \ref{th:zigbec}.
\end{proof}

\begin{prop}\label{prop:cone} The Darboux lines on
a cone, free of umbilic points,  can be integrated by quadratures. The
function
$$
{\mathcal I}(u,\alpha)=  k_g(u)\cos^3\alpha  $$
 is a
first integral  of the differential equation of Darboux lines.
 Here $k_g$ is the geodesic curvature of the intersection of the cone with the unitary sphere.

 Moreover, if $k_g^\prime/k_g<0 $ the ridge is zigzag. If
$k_g^\prime/k_g>0 $ the ridge is beak to beak.   See Fig.
\ref{fig:dcone}.
\begin{figure}[htb]
\begin{center}

 \includegraphics[scale=0.4]{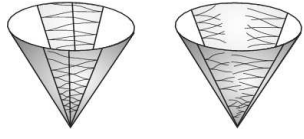}
\end{center}
\caption{ Darboux lines near regular ridges on a general cone
  \label{fig:dcone} }
\end{figure}

\end{prop}

\begin{proof} The cone can be parametrized by $X(u,v)=v\gamma(u)$
 where $|\gamma|=1$  and  $|\gamma^\prime|=1$ is a spherical curve.
We have  that $k_1(u,v)=k_g(u)$ and $k_2(u,v)=0$,
 since $\gamma^{\prime\prime} =-\gamma+k_g\gamma\wedge \gamma^\prime$
 and $N(u,v)=\gamma\wedge\gamma^\prime.$

The Darboux curves  are given by: $(\sin\alpha/\cos\alpha) d\alpha=\frac 13 (k_g^\prime/k_g) du$.
So it follows that ${\mathcal I}(u,\alpha)=k_g(u)\cos^3\alpha$ is the first integral.
The analysis of behavior near ridges follows from Theorem \ref{th:zigbec}.
\end{proof}

\begin{prop}\label{prop:rigdeLambda} The ridges on a surface of revolution
 are characterized by the singularities of the curve 
${\mathcal O}_1 \subset \Lambda^2 \subset \Lambda^4$ defined 
by the osculating circles of envelope of plane curves with
 center $(0,u)$ and radius $r(s)$.

Moreover the ridges correspond to ``vertices" of the curve
 $\gamma \subset \Lambda^2 \subset \Lambda^4$ defining the surface of revolution as an envelope.
\end{prop}

\begin{proof} The envelope of  plane curves with center
 $(0,u)$ and radius $r(u)$ is a plane curve $c(u)$ with curvature given
 by $k(u)=- \frac{ r^{\prime\prime} }{ 1-{r^\prime}^2-r r^{\prime\prime}}.$

According to equation \eqref{sig_sig} the curve 
${\mathcal O}_1(u)=k(u)m(u)+n(u)$ with ${\mathcal L}(m,m)=0$, ${\mathcal L}(n,n)=1$ and $n^\prime(u)=-k(u)m^\prime(u)$.
 So it follows that ${\mathcal O}_1^\prime =k^\prime(u) m(u)$ is a
 light vector ant is   zero precisely at the points where $k^\prime(u)=0$.

Let us remark that if a point of $\gamma(s)\in \gamma$ is a vertex
 then the radius $\rho(s)$ of the osculating circle ${\mathcal O}s)$
 is critical. Let us denote by $P(s)$ the osculating plane to $\gamma$
 at $\gamma(s)$, and by $x(s)\in P(s)$ the center of the osculating plane,
 which is also the (unique) critical point of the restriction of $\LL$ to $P(s)$.
 The Dupin necklace formed by ``the other osculating spheres" along the
 characteristic circle  $C(s)\subset \Sigma (s)$, where the sphere
 $\Sigma(s)$ correspond to the point $\gamma(s)\in \Lambda^4$ is the
 envelope of the spheres of the circle ${\mathcal O}s)$ and of the sphere
 of the associated circle ${\mathcal O}^*s)$. This circle is contained in
 the affine plane $Q(s)$, orthogonal to $span(O,P(s)$ and which contains 
the point $y(s)$ which belongs to the line spanned by $x(s)$ and satisfies
 $\LL(x(s),y(s))=1$. The square Lorentz norm $\LL(y(s_0))$ is critical
 when $\LL(x(s_0))$ and $\rho(s_0)$ are. Let $\sigma(s)$ be a curve of 
 osculating spheres of the other family along a line principal curvature.
  It is a light-like curve. Notice that, as $Q(s)$ is orthogonal
 to $y(s)$, $\LL(y(s),\sigma(s)= \LL(y(s))$. Derivating $\LL(y(s),\sigma(s))$
 with respect to $s$ we get, for a critical
 $s_0$, $\LL(y(s),\frac{d \sigma}{ds})|_{s=s_0}+\LL(\frac{d y}{ds},\sigma(s))|_{s=s_0}=0$.
 As $y(s)$ belongs to the 3-dimensional space $A=span({\mathcal O}s),O)$,
 which is independant of $s$, $\frac{dy}{ds}$ belongs also to $A$. All the
 affine spaces $Q(s)$ are orthogonl to $A$, therefore 
$\LL(\frac{d y}{ds},\sigma(s))|_{s=s_0}=\LL(\frac{d y}{ds},y(s))|_{s=s_0}=0$.
 Therefore $\LL(y(s),\frac{d \sigma}{ds})|_{s=s_0}=0$. This is possible if and
 only if $\frac{d \sigma}{ds}|_{s=s_0}=0$, as a Lorentz scalar product of
 non-zero time-like and light-like vector cannot be zero. The sphere $\sigma_0$ is
 therefore a ridge sphere. This reasonning is valid for all the curves in 
$\Lambda^4$ formed of osculating spheres of the other family along 
a line of principal curvature. The whole characteristic circle $C(s_0)$
 is therefore a ridge.  
\end{proof}

\begin{rema} Similarly in a cylinder $\alpha(u,v)=c(u)+v \overrightarrow{z}$,
 where $c$ is a plane curve with curvature $k$  the function 
$ {\mathcal I}(u,\alpha)=  k(u)\cos^3\alpha$ is a first integral of Darboux curves.
\end{rema}

\begin{rema} The geodesics on surfaces of revolution has a first
 integral given by  $ {\mathcal J}(u,\alpha)=h(u)\sin\alpha$.
 The parallel $u=u_0$ is a geodesic if and only if $h^\prime(u_0)=0$,
 and it is a hyperbolic geodesic when $h^{\prime\prime}(u_0)>0$. In
 a cone a first integral for the geodesics is given by
 $   {\mathcal J}(v,\alpha)=\cos\alpha/v$.
\end{rema}

\section{Darboux curves on quadrics}\label{sc:9}

The Darboux curves in the ellipsoid were considered in \cite{Pe}  by Pell.
 Here we complete and hopefully simplify his work.

The quadrics $\mathbb Q_{a,b,c}$ belongs to  the triple
orthogonal system of surfaces defined by the one parameter family
of quadrics, $\frac{x^2}{a-\lambda}+\frac{y^2}{b-
\lambda}+\frac{z^2}{c-\lambda}=1$ with $a>b>c>0$, see also
\cite{spivak}  and \cite{struik}.

Consider the principal chart $(u,v)$ and the parametrization of
 $\mathbb Q_{a,b,c}$ given by equation \eqref{eq:porto}.


For the ellipsoid  $u\in (b,a)$,\; $v\in (c,b)$ or $u\in (c,b)$,\; $v\in (b,a)$.

For the hyperboloid of one sheet $ u\in (b, a), \; v< c  \;\;\; \text{or } \;\;
  u<c,\;  v\in (b,a).$

For the hyperboloid of two sheets $ u\in (c, b), \; v< c  \;\;\; \text{or } \;\;
  u<c,\;  v\in (c,b).$

The first fundamental form of $\mathbb Q_{a,b,c}$  is given by equation
 \eqref{eq:Ie} and the second is given by
equation \eqref{eq:IIe}.





Therefore the principal curvatures are given by:

$$k_1=\frac{e}E=  \frac 1u \sqrt{\frac{abc}{uv}},\;\; k_2=\frac gG=\frac 1v
\sqrt{\frac{abc}{uv}}.$$

Also consider the functions
$$ r_1=\frac{\partial k_1}{\partial u}/(3(k_1-k_2))=\frac 12\frac{v}{u(u-v)},\;\;r_2=\frac{\partial k_2}{\partial v}/3((k_1-k_2))=\frac 12\frac{u}{v(u-v)}.$$

  The four umbilic points are $(±x_0,0,±z_0)= (±
 \sqrt{\frac{a(a-b)}{a-c}},0,± \sqrt{\frac{c(b-c)}{a-c}}\;).$

\begin{prop}
The differential equation of Darboux curves on a quadric $\mathbb Q_{a,b,c}$     is given by:

\begin{equation}\label{eq:deq} \aligned
 u^\prime=& \frac{1}{\sqrt{E}}\cos^2\alpha\sin\alpha,\;\;
  v^\prime=  \frac{1}{\sqrt{G}}\cos\alpha\sin^2\alpha\\
  \alpha^\prime=& \frac{r_1}{\sqrt{E}}\cos^3\alpha+\frac{r_2}{\sqrt{G}}\sin^3\alpha\\
\endaligned
\end{equation}

\end{prop}
\begin{proof} In a principal chart the differential equation of
Darboux curves is given by equation \eqref{eq:dpc}. Define $\tan\alpha=\frac{\sqrt{G}v^\prime}{\sqrt{E}u^\prime}$ 
so that $E{u^\prime}^2+G{v^\prime}^2=1$.
To obtain a regular extension of the differential equation
 \eqref{eq:dpc} to $\alpha=m\pi/2$ and consider it as a vector
 field in the variables $(u,v,\alpha)$  multiply the resultant
 equation by the factor $\cos\alpha\sin\alpha$ and the result follows.
\end{proof}

\begin{prop} The function
\begin{equation} \label{eq:ie} I(u,v,u^\prime,v^\prime)=I(u,v,\alpha)=\frac{\cos^2\alpha}{u}+
\frac{\sin^2\alpha}{v}, \;\; \tan \alpha=
\frac{v^\prime}{u^\prime} \sqrt{ \frac{G(u,v) }{E(u,v)
}}\end{equation} is a first integral of equation \eqref{eq:deq}.

\end{prop}

\begin{proof}
We have that  $I(u,v,\alpha)=\frac{\cos^2\alpha}{u}+\frac{\sin^2\alpha}{v}.$

We will show that $\frac{d}{ds}(I(u(s),v(s),\alpha(s) )=0$ along 
 a solution $(u(s),v(s),\alpha(s))$.

We have that
$$\aligned  E_s=& \frac
14\frac{uv^\prime}{H(u)}-\frac 14
\frac{u^\prime}{H(u)^2}[(u-2v)H(u)+(uv-u^2)H^\prime(u)] \\
G_s=& \frac 14\frac{vu^\prime}{H(v)}-\frac 14
\frac{v^\prime}{H(v)^2}[(u-2v)H(v)+(v^2-uv)H^\prime(v)]
\endaligned$$
Straightforward calculation leads to
$\frac{d}{ds}(I(u(s),v(s),u^\prime(s), v^\prime(s))) =0.$

\end{proof}

\begin{prop} The Darboux curves  on a quadric $\mathbb Q_{a,b,c}$  are the
real integral curves of the implicit differential equation:

$$ (v-\lambda) H(u){v^\prime}^2- (u-\lambda) H(v){u^\prime}^2=0.$$

 The normal curvature in a Darboux direction ${\mathcal D}$
 defined by $ I(u,v,dv/du )=1/\lambda $ is
given by $k_n(p,{\mathcal D})= \frac
1{\lambda}\sqrt{\frac{abc}{uv}}$.

This differential equation is equivalent to
$$k_n(u,v,[du:dv])= \frac{e(u,v)du^2+g(u,v)dv^2}{E(u,v)du^2+G(u,v)dv^2} =\frac
1{\lambda}\sqrt{\frac{abc}{uv}}=\frac
1{\lambda}   {(abc)}^{1/4}\, {\mathcal K}^{1/4}.  $$
\end{prop}

\begin{proof}
From the  equation $I=1/\lambda $ it follows that
$$(\frac{dv}{du})^2= \frac{(u-\lambda)}{(v-\lambda)}\frac{vE(u,v)}{u
G(u,v)}=\frac{(u-\lambda)H(v)}{(v-\lambda)H(u)}.$$

Here $dv/du$ is a Darboux direction $\mathcal D$.

 As
$k_n=\frac{e+g (dv/du)^2}{E+G(dv/du)^2}$ it follows from the
equation above and from equations \eqref{eq:Ie} and \eqref{eq:IIe}
that $k_n(p,{\mathcal D})=\frac 1{\lambda}\sqrt{\frac{abc}{uv}}= \frac
1{\lambda}   {(abc)}^{1/4}\, {\mathcal K}^{1/4}, \; {\mathcal K}=k_1k_2.$

For the reciprocal part consider the implicit differential
equation
$$  \frac{e(u,v)du^2+g(u,v)dv^2}{E(u,v)du^2+G(u,v)dv^2} =\frac
1{\lambda}\sqrt{\frac{abc}{uv}}. $$

Using equations \eqref{eq:Ie} and \eqref{eq:IIe} it follows that
this equation is equivalent to the following.

$$\frac{\lambda -u}{H(u)} du^2 - \frac{\lambda -v}{H(v)}dv^2=0
\,\;\; \Leftrightarrow \;\; (u-\lambda) H(v)
du^2-(v-\lambda)H(u)dv^2=0.$$
 The restriction in the
values of $\lambda$ is in order to consider only real solutions of
the implicit differential  equation obtained.
\end{proof}

\begin{prop} \label{prop:re1f2f} The ridge set of the quadric ${\mathbb Q}_{a,b,c}$ is
the intersection of the quadric with the   coordinates planes. Moreover:

\noindent   a)\;\; For the ellipsoid with $0<c<b<a$ it follows that:
\begin{enumerate}
\item[i)] The ellipse $E_{xy}=\{z=0\}\cap {\mathbb E}_{a,b,c}$,
 respectively $E_{yz}=\{x=0\}\cap {\mathbb E}_{a,b,c}$,  is a ridge
corresponding to $k_2$, respectively to $k_1$, and is zigzag.

\item[ii)] The ellipse $E_{xz}=\{y=0\}\cap {\mathbb E}_{a,b,c} $ containing
the four umbilic points $(±x_0, 0,±z_0)$ is the union of
ridges of $k_1$ and $k_2$. For $|x|>x_0$ the ridge correspond to
$k_1$. The ellipse $E_{xz}$    is beak to beak in both cases.
\end{enumerate}

\noindent   b)\;\;  For the hyperboloid of one sheet with $c<0<b<a$ it follows that:
\begin{enumerate}
\item[i)] The hyperbole $H_{yz} $ is a ridge corresponding
to $k_1$ and is beak to beak.

 \item[ii)] The hyperbole  $H_{xz} $ is a ridge corresponding to $k_1$ and is
zigzag.

\item[iii)] The ellipse $E_{xy} $ is a ridge
corresponding $k_2$    and it is  zigzag.

\end{enumerate}

\noindent   c)\;\; For the hyperboloid of two sheets with $c<b<0<a$ it follows that:

\begin{enumerate}

 \item[i)] The hyperbole $H_{xz} $ is a ridge corresponding
to $k_1$ and is zigzag.

\item[ii)] The hyperbole  $H_{xy} $
containing the four umbilic points $(±x_0,±y_0, 0 )$ is the
union of ridges of $k_1$ and $k_2$. For $|x|>x_0$ the ridge
correspond to $k_1$ and all segments of hyperbolas are beak to beak.

\end{enumerate}

\end{prop}

\begin{proof} a) Ellipsoid:
By symmetry is clear that the points of intersection
 of the coordinates planes with the ellipsoid are ridges points.
The principal curvatures have no critical points along the
corresponding principal curvature line in the complement of these
three ellipses. In a principal chart $(u,v)$ we have that $
\frac{dk_1}{du} =-\frac{3}{2u} k_1\ne 0$ and $ \frac{dk_2}{dv} =-\frac{3}{2v} k_2\ne 0$.

 It will be sufficient to check the condition of zigzag or beak to beak
 only in a point of a connected component of the ridge set.

 Consider the point $p_0=(-\sqrt{a},0,0)$. The ellipsoid
is parametrized by:
 $$x(y,z)=-\sqrt{a}+\sqrt{a} [\frac{
y^2}{2b}+\frac{z^2}{2c}+\frac{y^4}{8b^2}+\frac{y^2z^2}{4bc}+\frac{z^2}{8c^2}+
h.o.t.]$$

Therefore $k_1(p_0)=\sqrt{a}/b,$ \; $k_2(p_0)=\sqrt{a}/c$, $\;A=
3\sqrt{a}/b^2$ (A is the coefficient of $y^4$)  and
($(A-3k_1^3)(k_2-k_1)= -3a(a-b)(b-c)/(b^4c) <0$. Therefore the
ellipse $E_{xz}$ is beak to beak.

Also, let $E= 3\sqrt{a}/c^2 $ (coefficient of $ z^4$). So,
$(E-3k_2^3)(k_1-k_2)=3a(a-c)(b-c)/(bc^4)>0$. Therefore, by Theorem \ref{th:zigbec}  the ellipse
$E_{xy}$   is zigzag.

Now consider the point $q_0=(0,-\sqrt{b},0)$. The ellipsoid is
parametrized by:
 $$y(x,z)=-\sqrt{b}+  \sqrt{b} [\frac{
x^2}{2a}+\frac{z^2}{2c}+\frac{x^4}{8a^2}+\frac{x^2z^2}{4ac}+\frac{z^2}{8c^2}+
h.o.t.]$$

Now, with $A$ coefficient of $x^4$ and $E$ coefficient of $z^4$ it
follows that $(A-3k_1^3)(k_2-k_1)=3b(a-b)(a-c)/(a^4c)>0$ and
$(E-3k_2^3)(k_1-k_2)=3b(b-c)(a-c)/(ac^4)>0$.

So both  ellipses $E_{yz}$
and $E_{xy}$ are both zigzag, one for the corresponding  principal curvature.

\noindent b)\;  Hyperboloid of one sheet:
the ridges are given by
the  intersection of the hyperboloid   with the
coordinates planes.

Consider the point $q_0=(0,-\sqrt{b},0)$. The hyperboloid  is
parametrized by:
 $$y(x,z)=-\sqrt{b}+  \sqrt{b} [\frac{
x^2}{2a}+\frac{z^2}{2c}+\frac{x^4}{8a^2}+\frac{x^2z^2}{4ac}+\frac{z^2}{8c^2}+
h.o.t.]$$
 Therefore $k_1(q_0)=\sqrt{b}/a>0$ \;
$k_2(q_0)=\sqrt{b}/c<0$, $\;A=  3\sqrt{b}/a^2$ (A is the
coefficient of $x^4$)   and $E=3\sqrt{b}/c^2$ coefficient of $z^4$
it follows that $(A-3k_1^3)(k_1-k_2)=-3b(a-b)(a-c)/(a^4c)>0$ and
$(E-3k_2^3)(k_2-k_1)=-3b(b-c)(a-c)/(ac^4)<0$. The hyperbole $H_{yz}$
is beak to beak and the ellipse $E_{xy}$ is   zigzag.

Next consider the point $p_0=( -\sqrt{a},0,0)$. The hyperboloid is
parametrized by:
 $$x(y,z)= -\sqrt{a}+\sqrt{a} [\frac{
y^2}{2b}+\frac{z^2}{2c}+\frac{y^4}{8b^2}+\frac{y^2z^2}{4bc}+\frac{z^2}{8c^2}+
h.o.t.]$$

Therefore $k_1(p_0)=\sqrt{a}/b>0$ \; $k_2(p_0)=\sqrt{a}/c<0$,
$\;A=  3\sqrt{a}/b^2$ (A is the coefficient of $y^4$) $E=
3\sqrt{a}/c^2 $ (coefficient of $ z^4$) it follows that
($(A-3k_1^3)(k_1-k_2)= 3a(a-b)(b-c)/(b^4c) <0$ and
$(E-3k_2^3)(k_2-k_1)=-3a(a-c)(b-c)/(bc^4)<0$.  Therefore the
hyperbole  $H_{xz}$ and  the ellipse $E_{xy}$  are both  zigzag.

\noindent c)\; Hyperboloid of two sheets:
 the ridges are   the intersection
of the hyperboloid with the coordinates planes.

Consider the point $p_0=(  \sqrt{a},0,0)$. One leaf of the
  hyperboloid is parametrized by:
 $$x(y,z)=  \sqrt{a}-\sqrt{a} [\frac{
y^2}{2b}+\frac{z^2}{2c}+\frac{y^4}{8b^2}+\frac{y^2z^2}{4bc}+\frac{z^2}{8c^2}+
h.o.t.]$$

Therefore $k_1(p_0)=-\sqrt{a}/b>0$ \; $k_2(p_0)=-\sqrt{a}/c>0$,
$\;A= -3\sqrt{a}/b^2$ (A is the coefficient of $y^4$)  and
($(A-3k_1^3)(k_1-k_2)=   3a(a-b)(a-c)/(b^4c) <0$. Therefore the
hyperbole  $H_{xz}$ is  zigzag.

Also, let $E=  -3\sqrt{a}/c^2 $ (coefficient of $ z^4$). So,
$(E-3k_2^3)(k_2-k_1)=-3a(a-c)(b-c)/(bc^4)>0$. Therefore the
hyperbole $H_{xy}$ for $|x|<x_0$  is beak to beak.

For $|x|>x_0$ on the hyperbole $H_{xy}$ a similar analysis shows that
it is beak to beak.
\end{proof}

\begin{prop} \label{prop:de} Consider the ellipsoid ${\mathbb E}_{a,b,c}$
with $a>b>c>0$.
\begin{enumerate}
\item[i)] For $c< \lambda <b$ the Darboux curves  are and contained
in cylindrical  region $ c< v<\lambda  $ and the behavior is as in
Fig. \ref{fig:elipsoide}, upper left.

\item[ii)] For $\lambda=b$ the Darboux curves  are the circular
sections of the ellipsoid. These circles are contained in planes
parallels to the tangent plane to  $\mathbb E_{a,b,c} $ at the
umbilic points. These circles are tangent along the ellipse $E_y$
and through each  umbilic point pass only one Darboux curve. See
Fig. \ref{fig:elipsoide}, bottom right.

\item[iii)] For $b <\lambda < a $ the Darboux curves  are bounded
  contained in the two cylindrical region $ \lambda \leq u \leq
a$ and the behavior is as shown in the  Fig. \ref{fig:elipsoide}, upper right.

\end{enumerate}

\end{prop}

\begin{figure}[htbp]
\begin{center}
 \includegraphics[scale=0.5]{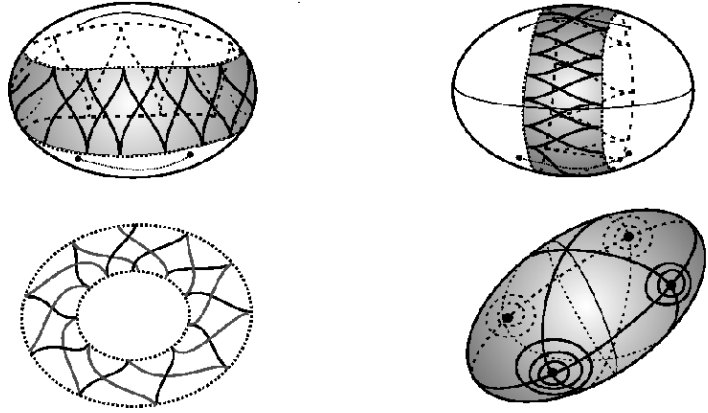}
   \caption { \label{fig:elipsoide}   Darboux curves  on the
ellipsoid.}
   \end{center}
 \end{figure}

\begin{proof} First case: $c<\lambda< b$.

The differential equation of Darboux curves  is given by:
$$  \frac{(v-\lambda)} {H(v)}{v^\prime}^2- \frac{(u-\lambda)}{ H(u)}{u^\prime}^2=0,
 \;\; c\leq v \leq \lambda   \;\; \text{ and}  \;\;b \leq u\leq a.$$

Define $d\sigma_1= \sqrt{(u-\lambda)/H(u)}du$ and $d\sigma_2=
\sqrt{(v-\lambda)/H(v)}dv$.

Therefore the differential equation is equivalent to
$d\sigma_1^2-d\sigma_2^2=0 $ , with $(\sigma_1, \sigma_2)\in
 [0,L_1 ]×[0,L_2 ]$ ( $L_1=\int_{b }^{a }
d\sigma_1 <\infty $, $\; L_2=\int_{c }^{\lambda } d\sigma_2
<\infty$ ).

In the ellipsoid this analysis implies the following.

The cylindrical region  $C_\lambda= \alpha([b,a]×
[c,\lambda])$ is foliated by the integral curves of an implicit
differential equation having cusp singularities in $\partial
C_\lambda$. We observe that this region is free of  umbilic point
and is bounded by principal curvature lines, in coordinates
defined by $v=c$ and $v=\lambda$.

The case $ b<\lambda <a$ the analysis is similar. Now the
differential equation of Darboux curves  are defined in the region
$[\lambda,a]×[c,b]$ and we have a cylindrical region
$C_\lambda= \alpha([ \lambda, a ]×[c,b])$. This

For $\lambda=b$ the differential equation  can be  simplified in
the following.

$(u-a)(u-c)dv^2-(v-a)(v-c)du^2=0. \;\; $

This equation is well defined in the rectangle $[c,a]×
[c,a]$.

Define $d\sigma_1=1/\sqrt { (u-a)(u-c) }du$ and
$d\sigma_2=1/\sqrt{(v-a)(v-c)}dv$. So the equation is equivalent
$d\sigma_1^2-d\sigma_2^2=0 $ with  $(\sigma_1, \sigma_2) \in [0,L
]×[0,L ]$ ($L=\int_{c}^{a } d\sigma_1$). So in this
rectangle all solutions are straight lines. The images of this
family of curves on the ellipsoid are  its  circular sections. In
fact we know that the ellipsoid has circular sections parallel to
the tangent planes at umbilic points. As the circles are always
Darboux lines it follows that the solutions of the differential
equation is the family of circular sections.  So we have two
families of circles having tangency along the ellipse $E_y$.
\end{proof}

\begin{prop} \label{prop:der} Consider an  ellipsoid \;$\mathbb E_{a,b,c}
$ with three axes $a>b>c>0$ and suppose $b <\lambda < a$.  Let
$L_1:= \int_{b}^{a} \sqrt{E(u,b)}du$ and  $L_2:=
\int_{c}^{\lambda} \sqrt{G( b,\lambda)}du$ and  define
$\rho=\frac{L_2}{L_1}$. Consider the Poincaré map $\pi: \Sigma
\to \Sigma$ associated to the foliation of Darboux curves  defined
by the implicit differential equation $I=1/\lambda$.

Then if $\rho\in\mathbb R\setminus \mathbb Q$ {\rm (} resp. $\rho\in \mathbb
Q${\rm )}
 all orbits are recurrent  {\rm (} resp. periodic{\rm )} on the cylinder region   $  v\leq \lambda\}$.
See Figure \ref{fig:elipsoide}, bottom left.
\end{prop}

\begin{proof}
The differential equation of Darboux curves  is given by:
$$  \frac{(v-\lambda)} {H(v)}{v^\prime}^2- \frac{(u-\lambda)}{ H(u)}{u^\prime}^2=0,
 \;\; c \leq v < \lambda <u\leq a.$$
Define $d\sigma_1=\sqrt { \frac{(u-\lambda)}{ H(u)} }du$ and
$d\sigma_2=\sqrt{\frac{(v-\lambda)} {H(v)}}dv$. By integration,
this leads to the chart $(\sigma_1, \sigma_2)$, in a rectangle
$[0,L_1]×[0,L_2]$ in which the differential equation of
Darboux is given by
$$d\sigma_1^2-d\sigma_2^2=0.$$
The proof ends with the analysis of the rotation number of the above equation.
 See similar analysis in  Propositions \ref{prop:55ma} and \ref{prop:der}.
\end{proof}

\begin{prop} \label{prop:dh2f} Consider a connect component of a  hyperboloid of two sheets $\mathbb H_{a,b,c}
$ with   $a>0> b>c.$
\begin{enumerate}
\item[i)] For $\lambda <c$ the Darboux curves  are non bounded and
contained in the non bounded region $ v<\lambda  $ and the
behavior is as in the Fig. \ref{fig:hp2f},   left.

\item[ii)] For $\lambda=c$ the Darboux curves  are the circular
sections of the hyperboloid. See Fig. \ref{fig:hp2f}, center.

\item[iii)] For $c <\lambda < b $ the Darboux curves  are non
bounded and contained in the   cylindrical region $ \lambda \leq
u \leq b$ and the local behavior   is as shown in Fig. \ref{fig:hp2f}, right.
\end{enumerate}
\end{prop}

\begin{figure}[htbp]
\begin{center}
\includegraphics[scale=0.65]{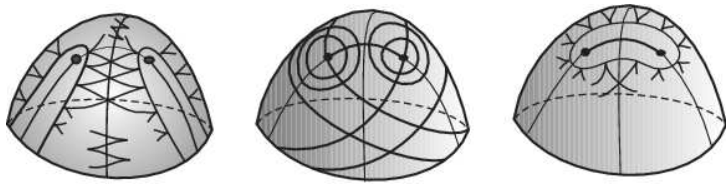}
   \caption { \label{fig:hp2f}  Darboux curves  on a connected component of a
hyperboloid of two sheets.}
   \end{center}
 \end{figure}

\begin{proof} The  analysis developed in the case of the ellipsoid
also works here. See proof of Proposition \ref{prop:de}.
\end{proof}

\begin{prop} \label{prop:dh1f} Consider an  hyperboloid of one sheet \; $\mathbb H_{a,b,c}
$ with   $a>b> 0> c.$ Let $\lambda\in (-\infty,c) \cup (b,\infty)$.
\begin{enumerate}
\item[i)] For $\lambda <c$ the Darboux curves  are  bounded and
contained in the cylindrical region $ \lambda \leq v\leq c$ and
the behavior is as in Fig. \ref{fig:hp1f}, upper left.

\item[ii)] For $b < \lambda < a$    the Darboux curves  are  unbounded and
contained in the cylindrical region $ b\leq u\leq \lambda$ {\rm (}outside the hyperbola  $E_x${\rm)} and
the behavior is as in Fig. \ref{fig:hp1f}, upper right.

 \item[iii)] For $\lambda=c$  and $\lambda=a$ the Darboux curves  are straight lines of the hyperboloid. For $\lambda=b$ the solutions are not real. See Fig.
\ref{fig:hp1f}, bottom left.

\item[iv)] For $a< \lambda $   all    Darboux curves are regular {\rm(} helices{\rm)} and goes to $\infty$ in both directions. See
Fig. \ref{fig:hp1f}, bottom right.

\end{enumerate}

\end{prop}

\begin{figure}[htbp]
\begin{center}
\includegraphics[scale=0.4]{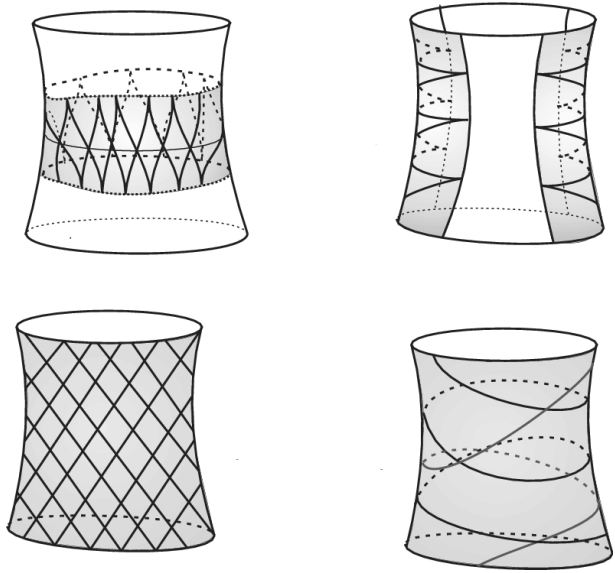}
   \caption { \label{fig:hp1f}  Darboux curves  on a
hyperboloid of one  sheet.}
   \end{center}
 \end{figure}

 \begin{proof} Similar to the proof of Proposition \ref{prop:de}.
 \end{proof}

  \begin{rema} The global behavior of geodesics in quadrics, in particular in the ellipsoid, was studied in \cite{gascoloquio}.\end{rema}

\vskip .5cm
\noindent{\bf Acknowledgements} The first author is grateful to the Faculty and staff of Institut de Mathé­ma­tiques de Bourgogne for the hospitality
during his stay in Dijon.  The  first author is fellow of  CNPq and performed  this
  work  under the project CNPq 473747/2006-5 and FUNAPE/UFG.
  The third   author was supported by the European Union grant, no. ICA1-CT-2002-
70017.

\bibliographystyle{plain}

 \newpage

\author{\noindent  Ronaldo Garcia \\Instituto de Matemática e Estat\'{\i}stica,\\
Universidade Federal de Goiás,
\\CEP 74001-970, Caixa Postal
131,\\Goiânia, GO, BRAZIL \\
e-mail: ragarcia@mat.ufg.br\\
\\
\\ Rémi Langevin\\Institut de Mathématiques de Bourgogne,
\\U.F.R. Sciences et Techniques\\9, avenue Alain Savary\\
Université de Bourgogne,  B.P. 47870\\
 21078 - DIJON Cedex, FRANCE \\
e-mail: Remi.Langevin@u-bourgogne.fr\\
 \\
 Pawel Walczak \\ 
 Katedra Geometrii,  Wydzia\l \/  Matematyki\\
 Uniwersytet \L ódzki\\
 ul. Banacha 22, 90-238 \L ód\'z  POLAND\\
e-mail: pawelwal@math.uni.lodz.pl}

\end{document}